\documentclass[final]{siamart}

\usepackage{amsmath,amssymb}
\usepackage{amsmath,bm}
\usepackage{amsfonts}
\usepackage{dsfont}
\usepackage{xr}
\usepackage{enumitem}
\newtheorem{remark}{remark}[section] 

\newcommand{\argmin}{\operatornamewithlimits{argmin}}
\newcommand{\bydef} {{\buildrel{\triangle}\over =}}

\renewcommand{\theenumi}{(\Roman{enumi})}

\definecolor{red}{rgb}{1,0,0}
\definecolor{Purple}{rgb}{1,0,0}

\allowdisplaybreaks

\title{Mean Field Game Theory for Agents with Individual-State Partial Observations\thanks{Some of the work in this paper was presented at the $55^{th}$ IEEE Conference on Decision and Control, Las Vegas, NV, USA, December, 2016.}}
\author{Nevroz \c{S}en \thanks{ABB Inc. San Jose CA. This author's work was performed at the Center for Intelligent Machines (CIM) and the Department of Electrical and Computer Engineering, McGill University, Montreal QC, Canada (email: {\tt\small nsen@cim.mcgill.ca}). } \and  Peter E. Caines\thanks{Department of Electrical and Computer Engineering and CIM, McGill University, Montreal QC, Canada (email: {\tt\small peterc@cim.mcgill.ca}). Work supported by research grants to this author from Natural Science and Engineering Research Council of Canada (NSERC) and  Air Force Office of Scientific Research (AFOSR).}} %
\allowdisplaybreaks
\begin{document}
\maketitle
\thispagestyle{empty}
\pagestyle{empty}

\begin{abstract}
\noindent Subject to reasonable conditions, in large population stochastic dynamics games, where the agents are coupled by the system's mean field (i.e. the state distribution of the generic agent) through their nonlinear dynamics and their nonlinear cost functions, it can be shown that a best response control action for each agent exists which (i) depends only upon the individual agent's state observations and the mean field, and (ii) achieves a $\epsilon$-Nash equilibrium for the system. In this work we formulate a class of problems where each agent has only partial observations on its individual state. We employ nonlinear filtering theory and the Separation Principle in order to analyze the game in the asymptotically infinite population limit. The main result is that the $\epsilon$-Nash equilibrium property holds where the best response control action of each agent depends upon the conditional density of its own state generated by a nonlinear filter, together with the system's mean field. Finally, comparing this MFG problem with state estimation to that found in the literature with a major agent whose partially observed state process is independent of the control action of any individual agent, it is seen that, in contrast, the partially observed state process of any agent in this work depends upon that agent's control action.
\end{abstract}


\pagestyle{myheadings}
\thispagestyle{plain}
\markboth{NEVROZ \c{S}EN AND PETER E. CAINES}{MEAN FIELD GAMES WITH PARTIAL OBSERVATIONS}

\section{Introduction}\label{sec:intro}
For dynamical games of mean field type it has been demonstrated that when the agents are coupled through their dynamics and their cost functions, the best response control policies in the asymptotically infinite population limit only depends upon their individual state and the system mean field. Furthermore, such policies generate approximate Nash equilibria when they are applied to a large finite population game, see \cite{hmc-cdc03}, \cite{hmc-mfg-main}, \cite{hmc-cdc06} and \cite{hmc-tac} among others, by Huang, Malham\'{e} and Caines, \cite{lasry-lions-1}, \cite{lasry-lions-2} and \cite{lasry-lions-3}, by Lasry and Lions. 

A distinct consequence of such a result is that in the mean field games (MFG) set-up an individual agent does not have a significant benefit in learning the state of an other agent, and therefore the estimation of any other agent's state process has negligible value. Nonetheless, in practical situations one does not have access to complete observation of its own state and and therefore models of such (PO) MFG systems where the agents' controls depend upon the agents' observation processes can only represent them as functions of the agents' states via estimates of those states. Such a model for linear quadratic Gaussian (LQG) MFG type of problems has been considered in \cite{hmc-mtns06} and approximate Nash equilibrium is obtained on an extended state space. In this work, we consider the nonlinear MFG where an individual agent has noisy observation on its own state. 

Recent works, (\cite{minyi-major} and \cite{NP-Siam2013}), consider MFG involving a major agent and many minor agents (MM-MFG) where, by definition, a minor agent is an agent which, asymptotically as the population size goes to infinity, has a negligible influence on the overall system while the overall population's effect on it is significant, and where a major agent is the agent which has asymptotically non-vanishing influence on each minor agent as the population size goes to infinity. A fundamental feature of this setup is that, in contrast to the situation without a major agent, the mean field is stochastic due to the stochastic evolution of the state of the major agent and the best response processes of each minor agent depend on the state of the major agent. Motivated by this observation, state estimation problems in the nonlinear MFG with a major agent is considered in \cite{sen-caines-po-sicon} (see \cite{pec-ak-tac} and \cite{dena-pec-cdc} for the LQG case) where the major agent's state process is partially observed but the agents have complete observation of their own states. Adopting the approach of constructing an equivalent completely observed model via application of nonlinear filtering, the MFG problem is analyzed in the space of conditional densities and the existence of Nash equilibria in the infinite population and the $\epsilon$-Nash equilibria for the finite population game is obtained. We finally remark that in addition to \cite{minyi-major} and \cite{NP-Siam2013}, MFG setup with major and minor agents has also been considered in \cite{Bensoussan-dp} and \cite{carmona} where in the former the authors generalized the MM-MFG setup to the case where the mean field is determined by control policy of the major agent and in the later, a probabilistic approach is taken for MFG problems in which major agent's state exists in the both state and the cost functions. We also refer to \cite{carmona-book} for the analysis of MFG with common noise. 

The individuals dynamics in the infinite population limit in an MFG setup are charactherized by McKean-Vlasov (MV) type stochastic differential equations (SDEs). These SDEs have the property that the dynamics depend on the distribution of the state process. Hence, a PO stochastic optimal control problem (SOCP) is formulated for MV type SDEs and as a consequence, the filtering equations should first be developed for such SDEs for which a theory for joint state and distribution estimation in the case the measure is stochastic is developed in \cite{sen-caines-nlf-sicon}. Following the standard approach in the literature, once the filtering equations in the form of normalized or unnormalized densities are obtained, it is possible to obtain a form of the Hamilton-Jacobi-Bellman (HJB) equation in functional spaces. This is the path that we employ in the paper by using the unnormalized conditional densities. 

It is also worthwhile to provide a summary of the technical steps that one shall develop in a nonlinear PO MFG setup. We first remark that one can follow different approaches in order to prove the convergence properties of MFG in the infinite population. Among these, the convergence of the dynamics of the controlled state process into a MV type dynamics when feedback controls are applied, see \cite{hmc-mfg-main}, greatly simplifies the analysis of the associated optimal control problem. In the partially observed setup, we follow this approach consequently, as the first step, we shall prove such a convergence argument for the case where the control policies are in the feedback form for conditional densities. We next analyze the fixed point property on the Wasserstein space of probability measures. Recall however that the solution to the completely observed MFG problem is given by a coupled HJB and a Fokker-Plank-Kolmogorov (FPK) equation which essentially requires to analyze the sensitivity of the solutions to the HJB equations with respect to the probability measure representing the mean field. In the PO MFG one needs to generalize such sensitivity results with respect to the conditional density component representing the information state. This is achieved by using the robustness property of the nonlinear filter. In the final stage, we shall prove the approximate Nash equilibrium property of the best response control policies obtained as the solution to the HJB equation of the infinite population game. 

The organization of the paper is as follows. In Section \ref{sec:mfg-main} we present a MFG setup with uniform agents and discuss the main results in a brief manner. In Section\ref{sec:po-mfg} we formulate the state estimation problem and present the associated completely observed system via applying the separation principle. We also provide a solution in the form of HJB equation for the completely observed problem. In Section \ref{sec:soln-fo-socp} we prove the existence of a Nash equilibrium between an individual and the mass in the infinite population limit and in Section \ref{sec:e-nash}, we demonstrate the approximate Nash equilibrium property of the best response processes obtained in the former section. In Section \ref{sec:example} we present an example where the completely observed model has a finite dimensional information state and hence provides a more tractable MFG system. In Section \ref{sec:comp-maj} we briefly compare the results presented in the paper with a PO MM-MFG model. We conclude the paper with Section \ref{sec:conc}. 

Throughout the paper we use the following notation. For a matrix $A$, $A^T$, $\mathbf{tr}(A)$ and $A_{ij}$ denotes the transpose, the trace and the corresponding entry, respectively. $\nabla_x$ and $\nabla_{xx}^2$ denotes the gradient and Hessian operators with respect to the variable $x$ and in a one-dimensional domain, $\partial_x$ and $\partial_{xx}^2$ will be used instead. Let $\mathbb{S}$ be a metric space. Then, $\mathcal{B}(\mathbb{S})$ denotes the Borel $\sigma$-algebra and $\mathcal{P}(\mathbb{S})$ denotes the space of probability measures, respectively, on $\mathbb{S}$. Let $(\Omega,\mathcal{F},\{\mathcal{F}_t\}_{t\geq 0},P)$ be a filtered probability space satisfying usual conditions.  Conditional expectation with respect to a sigma algebra $\mathcal{F}$ is denoted by $\mathbb{E}\left(\cdot|\mathcal{F}\right)$. For an Euclidean space $H$, we denote by $L_{\mathcal{G}}^2\left([0,T];H\right)$ the set of all $\{\mathcal{G}\}$-adapted $H$-valued processes such that $\mathbb{E}\int_{0}^T|f(t,\omega)|^2dt<\infty$. 
\section{Mean Field Games with Uniform Agents}\label{sec:mfg-main}
We consider a stochastic dynamic game with $N$ agents, $\{\mathcal{A}_i, 1\leq i \leq N\}$, where the dynamics of the agents are given by the following controlled SDEs on $\bigl(\Omega,\mathcal{F},\{\mathcal{F}_t\}_{t\geq 0},P\bigr)$:
\begin{eqnarray}
dz_{i}^N(t)=\frac{1}{N}\sum_{j=1}^Nf\left(t,z_{i}^N(t),u_{i}^N(t),z_{j}^N(t)\right)dt+\sigma dw_{i}(t), \label{eq:sta-dyn}
\end{eqnarray}
with terminal time $T\in (0,\infty)$ and initial conditions ${z}_{i}^N(0)=z_i(0)$, $1\leq i \leq N$, where (i) $z_{i}^N(t) \in \mathbb{R}$, $u_{i}^N(t) \in U$, $0\leq t\leq T$, are the state and control input of agent $\mathcal{A}_i$, (ii) $f:[0,T]\times\mathbb{R}\times U\times \mathbb{R}\rightarrow \mathbb{R}$ is a measurable function, (iii) $(w_{i}(t))_{t\geq0}$ are independent standard Brownian motions in $\mathbb{R}$ and; (iv) $\sigma >0$ is constant. For $1\leq j \leq N$ we denote by $u_{-j}^N:=\{u_1^N,\ldots,u_{j-1}^N,u_{j+1}^N,\ldots,u_N^N\}$, where agents' states and controls are taken to be scalar valued for simplicity of notation throughout the paper. The objective of each agent is to minimize its cost-coupling function given by
\begin{eqnarray}
J_{i}^N(u_{i}^N,u_{-i}^N):=\mathbb{E} \hspace{-0.01in}\int_{0}^T\frac{1}{N}\sum_{j=1}^NL\left(z_{i}^N(t),u_{i}^N(t),z_{j}^N(t)\right)dt,\label{eq:cost}
\end{eqnarray}
where $L:\mathbb{R}\times U\times \mathbb{R}\rightarrow\mathbb{R}_{+}$. Remark that the above model can be generalized to the case where the diffusion coefficient depends on the mean field coupling, where the state processes take values in, say, $\mathbb{R}^m$ and where the cost functions are time varying. We assume the followings:
\begin{itemize}
\item [(A0)] The initial states $\{z_j(0), 1 \leq j \leq N\}$ are mutually independent, independent of all Brownian motions and satisfy $\sup_{j \in \{1,\ldots,N\}}\mathbb{E}|z_j(0)|^2\leq k < \infty$, where $k$ is independent of $N$. Furthermore, let $F_N(x):=\left(1/N\right)\sum_{i=1}^N \mathbf{1}_{\{\mathbb{E}z_i(0)\leq x\}}$ denote the empirical distribution of agents where $\mathbf{1}_{\{\mathbb{E}z_i(0)< x\}}=1$ if $\mathbb{E}z_i(0)< x$ and $\mathbf{1}_{\{\mathbb{E}z_i(0)< x\}}=0$ otherwise. Then we assume $\{F_N:N\geq 1\}$ converges to a distribution $F$ weakly.
\item[(A1)] $U$ is a compact set.
\item[(A2)] The functions $f(t,x,u,y)$, $L(x,u,y)$ are continuous and bounded in all their parameters and Lipschitz continuous in $(x,y)$. 
\item [(A3)] The first and second order derivatives of $f(t,x,u,y)$ and $L(x,u,y)$ with respect to $x$ are all uniformly continuous and bounded with respect to all their parameters and Lipschitz continuous in $y$. \item [(A4)] $f(t,x,u,y)$ is Lipschitz continuous in $u$.
\end{itemize}   
For the system described by (\ref{eq:sta-dyn})-(\ref{eq:cost}), the goal is to find individual control strategies and characterize their optimality with regard to Nash equilibrium. Following the standard approaches in literature, the asymptotic analysis ($N\rightarrow \infty$) of the above game shall be considered first and as a consequence, MV type dynamics approximating the state dynamics of an individual agent should be obtained. More explicitly, let $\phi\left(t,x\right) \in C_{\rm{Lip}(x)}\left([0,T]\times \mathbb{R};U\right)$, the space of $U$-valued, continuous functions on $[0,T]\times \mathbb{R}$ with Lipschitz coefficient in $x$, which is used by the agents as their control laws. Hence, the closed loop dynamics of agents are given by
\begin{eqnarray}
d\hat{z}_{i}^{o}(t)=\frac{1}{N}\sum_{j=1}^Nf\left(t,\hat{z}_{i}^{o}(t),\phi(t,\hat{z}_i^{o}),\hat{z}_{j}^{o}(t)\right)dt+\sigma dw_{i}(t), \label{eq:sta-dyn-clp}
\end{eqnarray}
for which unique solution is known to exist \cite[Chapter 1, Theorem 6.16]{xyz}. Consider the following MV type dynamics
\begin{eqnarray}
dz_{i}^{o}(t)=f\left[t,z_{i}^{o}(t),\phi\left(t,z_i^{o}(t)\right),\mu_t\right]dt+\sigma dw_{i}(t), \label{eq:sta-dyn-mv}
\end{eqnarray}
with $z_{i}^{o}(0)=z_i(0)$ and $\mu_t \in \mathcal{P}\left(\mathbb{R}\right)$. Here $f\left[t, x,u,\mu_t\right]=\int_{\mathbb{R}}f\left(t,x,u,y\right)\mu_t(dy)$ and we use the same notation in the rest of the paper. A pair $\left(z_i^o(t),\mu_t\right)$ for (\ref{eq:sta-dyn-mv}) is said to be a \textit{consistent solution} if $z_i^o(t)$ is a solution to the SDE in (\ref{eq:sta-dyn-mv}) and $P(z_i^o(t)\leq \alpha) = \int_{-\infty}^\alpha \mu_t(dy)$ for all $\alpha \in \mathbb{R}$ and $0\leq t \leq T$. The closed loop dynamics in (\ref{eq:sta-dyn-clp}) can be $O(1/\sqrt{N})$-approximated by the MV type dynamics given by (\ref{eq:sta-dyn-mv}) \cite{hmc-mfg-main}. We shall now proceed, following the outline summarized in the introduction, with the partially observed MFG in the infinite population and McKean-Vlasov approximation with filtering dependent control policies.
\section{Partially Observed Mean Field Games and Nonlinear Filtering for MV Systems}\label{sec:po-mfg}
In this section we formulate the estimation problem associated to the MFG set-up described above. Let agent $\mathcal{A}_i$ has access to a noisy observation of its own state via:
\begin{eqnarray}
dy_i(t)=h\left(t, z_i^o(t)\right)dt+d\nu_i(t),\label{eq:obs-dyn}
\end{eqnarray}
where $\left(\nu_i(t)\right)_{0\leq t \leq T}$ is a Brownian motion independent of $\big\{z_i(0), \left(w_i(t)\right)_{0\leq t \leq T}, 1\leq i \leq N\big\}$ and of the other noise processes $\big\{(\nu_{-i}(t)_{0\leq t \leq T} \big\}$. We assume the following.
\begin{itemize}
\item [(A5)] The function $h:[0,T]\times\mathbb{R}\rightarrow \mathbb{R}\in C^{1,2}_{t,x}\left([0,T]\times \mathbb{R}\right)$, the space of functions which are differentiable in $t$ and twice differentiable in $x$, with $|\partial_xh(t,x)|+|\partial^2_{xx}h(t,x)|\leq K$ and $|\partial_th(t,x)|\leq K (1+|x|)$ for all $(t,x) \in [0,T]\times\mathbb{R}$.
\end{itemize}
Following the standard approach to the PO SOCP, we shall construct the associated completely observed system via application of nonlinear filtering for the dynamics described in (\ref{eq:sta-dyn-mv}). But prior to that we obtain an MV type approximation result for the state process controlled with filtering dependent policies, since under suitable assumptions the optimal control takes a feedback form given by the solution of an HJB equation with infinite dimensional domain.
\subsection{Nonlinear Filtering for McKean-Vlasov Dynamics}
The nonlinear filtering equations that each agent needs to generate are defined as follows: Given the history of observations $\mathcal{F}_t^{y_i}:=\sigma\{y_i(s):s\leq t\}$, determine a recursive expression for $\mathbb{E}\left[\ell\left(z_i^o\left(t\right)\right)|\mathcal{F}_t^{y_i}\right]$ for $\ell \in C_b^2(\mathbb{R})$, the space of all bounded differentiable functions with bounded derivatives up to order 2. Note that the agent's state $z_i^o(t)$ has MV type dynamics and so, for a fixed measure flow, we have
\begin{eqnarray}
f\left[t,z_i^o,u_i,\mu_t\right]=\int_{\mathbb{R}}f[t,z_i^o,u_i,x]\mu_t(dx):=f^*\left(t,z_i^o,u_i\right),\label{eq:sta-dyn-indc}
\end{eqnarray}
where $f^*:[0,T] \times \mathbb{R} \times U\rightarrow \mathbb{R}$. Hence, consider the SDEs
\begin{eqnarray}
dz_i^o(t)&=&f^*\left(t,z_i^o(t),u_i(t)\right)dt +\sigma dw_{i}(t),\label{eq:mfg-obser-ind-meas}\\
dy_i(t)&=&h\left(t,z_i^o(t)\right)dt+ d\nu_i(t) \label{eq:obs-dyn-sta-meas-b}.
\end{eqnarray}
The filtering problem for the MV system described by (\ref{eq:mfg-obser-ind-meas})-(\ref{eq:obs-dyn-sta-meas-b}) has been analyzed in \cite{sen-caines-nlf-sicon} where filtering equations generating conditional distributions are obtained. We can similarly obtain filtering equations in the form of conditional densities as follows. Define the following innovation process: $I_i(t)=y_{i}(t) - \int_{0}^t \mathbb{E}\left[h(s,z_i^o(s))|\mathcal{F}_s^{y_i}\right]ds$ which can be shown to be a $\mathcal{F}_t^{y_i}$-Brownian motion under the measure $P$. Let $\pi_i:=P\left(z_i^o(t)|\mathcal{F}_t^{y_i}\right)$ and define
\begin{eqnarray}
\mathcal{L}\ell:=\frac{1}{2}\sigma^2\partial_{xx}^2\ell+f^*\partial_x\ell \label{eq:operator-sta}.
\end{eqnarray}
Define the adjoint operator on $ C^2\left(\mathbb{R}\right)$ as:
\begin{eqnarray}
\mathcal{L}^\ast\theta(x)= \frac{1}{2}\partial_{xx}^2\sigma\theta(x)-\partial_xf^\ast\theta(x).\label{eq:adjoint-op}
\end{eqnarray}
Let $\varphi_{i}$ denote the probability density for $\pi_i$ i.e., for $A \in \mathcal{B}(\mathbb{R})$, $\pi_{i}(t,A)=\int_{A}\varphi_{i}(t,x)dx$ where $\varphi_{i}(\cdot)$ is $(t,x)$-measurable and $\mathcal{F}_t^{y_i}$ adapted for each $t \in [0,T]$. Then $\varphi_{i}(t,x)$ satisfies the following: For every $t$,
\begin{eqnarray}
&&\varphi_{i}(t,x)=\varphi_i(0,x)\label{eq:dens-filt-main-norm}\\
&&+\int_{0}^t\mathcal{L}^\ast\varphi_i(s,x)ds+\int_{0}^t\varphi_i(s,x)\biggl\{h(s,x)-\int_{\mathbb{R}}h(s,x')\varphi_i(s,x')dx'\biggr\}dI_i(s),\nonumber
\end{eqnarray}
for a.e. $x$ with probability $1$ and where $\varphi_i(0,x)$ is the initial conditional density and $I_i(t)$ is the Innovations process, which is a Brownian motion, defined above. This can be shown, for instance, by following \cite[Theorem 11.2.1]{Kallianpur}. Based on the consistency based approach to MFG \cite{hmc-mfg-main}, we now provide a decoupling result which demonstrates that the closed loop dynamics of each agent in the infinite population limit is approximated by MV SDEs in the partially observed setup. 
\subsection{McKean-Vlasov Approximation with Partial Information}
Let $\mathsf{E}$ be a vector space with norm $\| \cdot \|_{\mathsf{E}}$ such that the process $\varphi_{i}(t)$, $0\leq t \leq T$, satisfying (\ref{eq:dens-filt-main-norm}) takes values in. Recall also that the process $\varphi_{i}(t)$ is $\mathcal{F}_t^{y_i}$-adapted. Let $\alpha(t,p):[0,T]\times \mathsf{E}\rightarrow U$ be an arbitrary measurable process and assume that
\begin{enumerate}
\item [(M1)] $\alpha(t,p) \in C_{\rm{Lip}(p)}\left([0,T]\times \mathsf{E};U\right)$ and $\alpha(t, 0) \in L^2_{\mathcal{F}_t^{y_i}}([0,T];U)$.
\end{enumerate}
\noindent Assume that the process $\alpha(t,\cdot)$ is used by agent $i$ as its control laws in (\ref{eq:sta-dyn}) such that $u_i=\alpha$ for $1\leq i \leq N$. We then obtain the following closed-loop dynamics: 
\begin{eqnarray}
dz_{i}^N(t)&=&\frac{1}{N}\sum_{j=1}^Nf\left(t,z_{i}^N(t),\alpha\left(t,\varphi_{i}(t)\right),z_{j}^N(t)\right)dt+\sigma dw_{i}(t), ~z_{i}^N(t)=z_i(0). \label{eq:sta-dyn-cl}
\end{eqnarray}
One can show that under the assumptions (A1)-(A4), the systems of equation given in (\ref{eq:sta-dyn-cl}) has a unique solution $\left(z_{1}^N,\ldots,z_{N}^N\right)$ by following similar steps to those in the proof of Theorem 6.16 of \cite[p. 49]{xyz} and by using the robustness (i.e., continuity with respect to the observation path) of nonlinear filter; see Theorem \ref{the:mv-sde-uniq}. We now introduce the MV system for the generic agent where the agent's MV system shall contain the estimation of its own state via nonlinear filtering equations:
\begin{eqnarray}
d\hat{z}(t)&=&f\left[t,\hat{z}(t),\alpha\left(t,\varphi(t)\right),\mu_t
\right]dt+ \sigma dw(t), \label{eq:mv-appr-1}\\
dy(t)&=&h\left(\hat{z}(t)\right)dt+ d\nu(t),~\enspace 0 \leq t \leq T, \label{eq:mv-appr-obs-1}
\end{eqnarray}\noindent with the initial condition $\hat{z}(0)=z(0)$ and $\bigl(w(t),\nu(t)\bigr)_{0\leq t \leq T}$ are standard Brownian motion in $\mathbb{R}$, which are independent of each other and independent of initial condition $z(0)$. Furthermore, we characterize $\mu_t$ by $P\left(\hat{z}(t)\leq \alpha\right)=\int_{-\infty}^\alpha\mu_t(dx)$, $0< t \leq T$. Finally, $\varphi(t)$ is the $\mathcal{F}_t^y$-adapted solution to filtering equation for the conditional density. We remark that under (A0)-(A4), (A5) and (M1) it can be shown that a unique consistent solution to the above MV system exists; see Theorem \ref{the:mv-sde-uniq}. Let us also introduce
\begin{eqnarray}
d\hat{z}_i(t)&=&f\left[t,\hat{z}_i(t),\alpha\left(t,\varphi_i(t)\right),\mu_t\right]dt+ \sigma dw_i(t),  \label{eq:mv-appr-min-2}\\
dy_i(t)&=&h\left(t,\hat{z}_{i}(t)\right)dt+ d\nu_i(t),~0 \leq t \leq T, \label{eq:mv-appr-obs-2}
\end{eqnarray}
where $\left(w_i(t),\nu_i(t)\right)_{0\leq t \leq T} \enspace 1 \leq i \leq N$ Brownian motions in $\mathbb{R}$ which are are independent of each other and independent of $(z_i(0), 1\leq i \leq N)$ and $\mu_t$ is the law of $\hat{z}_i(t)$. These equations can be considered as $N$ independent copies of (\ref{eq:mv-appr-1})-(\ref{eq:mv-appr-obs-1}). We can now state the MV approximation result.
\begin{theorem}\label{the:mv-appr-partial}
Assume (A0)-(A4), (A5) and (M1) hold. Then 
\begin{eqnarray}
\sup_{1\leq j \leq N} \sup_{0\leq t \leq T} \mathbb{E}|z_j^N(t)-\hat{z}_j(t)|=O\left(\frac{1}{\sqrt{N}}\right) \label{eq:mv-conv-partial},
\end{eqnarray}
where $z_j^N(t)$ and $\hat{z}_j(t)$, $1\leq j \leq N$, are given in (\ref{eq:sta-dyn-cl}) and (\ref{eq:mv-appr-min-2}), respectively, and $O(\frac{1}{\sqrt{N}})$ depends on $T$.
\end{theorem}
\begin{proof}
The proof is an extension of \cite[Theorem 12]{hmc-mfg-main} to the case where control laws depend on the filtering processes. Consider first the $i$th agent and notice that
\begin{eqnarray}
&&z^N_i(t)-\hat{z}_i(t)=\label{eq:mvprf1}\\
&&\int_{0}^t\frac{1}{N}\sum_{j=1}^Nf\left(t,z_{i}^N, \alpha\left(s,\varphi_{i}(s)\right),z_j^N\right)ds-\int_{0}^tf\left[s,\hat{z}_i(s),\alpha\left(s,\varphi_i(s)\right),\mu_s\right]ds\nonumber. 
\end{eqnarray}
Let 
\begin{eqnarray}
&&D_{i}(s):=\label{eq:mvprf2}\\
&&\frac{1}{N}\sum_{j=1}^Nf\left(s,z_{i}^N(s), \alpha\left(s,\varphi_{i}(s)\right),z_{j}^N(s)\right)
-\int_{\mathbb{R}}f\left(s,\hat{z}_i(s),\alpha\left(s,\varphi_i(s)\right),y\right)\mu_s(dy), \nonumber
\end{eqnarray}
and observe that
\begin{eqnarray}
D_{i}(s)&=&D_{i}^1(s)+D_{i}^2(s)+D_{i}^3(s),\nonumber\\
D_{i}^1(s)&:=&\frac{1}{N}\sum_{j=1}^Nf\left(s,z_{i}^N(s), \alpha\left(s,\varphi_{i}(s)\right),z_{j}^N(s)\right)-\frac{1}{N}\sum_{j=1}^Nf\left(s,\hat{z}_{i}(s), \alpha\left(s,\varphi_{i}(s)\right),z_{j}^N(s)\right),\nonumber\\
D_{i}^2(s)&:=&\frac{1}{N}\sum_{j=1}^Nf\left(s,\hat{z}_{i}(s), \alpha\left(s,\varphi_{i}(s)\right),z_{j}^N(s)\right)
-\frac{1}{N}\sum_{j=1}^Nf\left(s,\hat{z}_{i}(s), \alpha\left(s,\varphi_{i}(s)\right),\hat{z}_{j}(s)\right),\nonumber\\
D_{i}^3(s)&:=&\frac{1}{N}\sum_{j=1}^Nf\left(s,\hat{z}_i(s), \alpha\left(s,\varphi_{i}(s)\right),\hat{z}_{j}(s)\right)-\int_{\mathbb{R}}f\left(s,\hat{z}_i(s),\alpha\left(s,\varphi_i(s)\right),y\right)\mu_s(dy)\nonumber.
\end{eqnarray}
By the Lipschitz continuity of $f$ and $\alpha$, there exists a constant $C>0$ independent of $N$ such that
\begin{eqnarray}
|D_{i}^1+D_{i}^2|\leq C \sum_{j=1}^N(1/N)\left(|z^N_i-\hat{z}_i| + |z^N_j-\hat{z}_j|\right)\label{eq:mvprf4}.
\end{eqnarray}
From (\ref{eq:mvprf1})-(\ref{eq:mvprf4}), it follows that
\begin{eqnarray}
&&\sup_{0\leq s \leq t}\left| z^N_i(s)-\hat{z}_i(s)\right| \leq C \int_{0}^t\left| z^N_i(s)-\hat{z}_i(s)\right|ds\nonumber\\
&&\hspace{0.5in}+C \int_{0}^t(1/N)\sum_{j=1}^N\left| z^N_j(s)-\hat{z}_j(s)\right|ds+\int_{0}^tD_{i}^3(s)ds\label{eq:mvprf5}
\end{eqnarray}
which gives
\begin{eqnarray}
&&\sum_{i=1}^N\sup_{0\leq s \leq t}\left| z^N_i(s)-\hat{z}_i(s)\right| \nonumber\\
&&\leq 2C\sum_{i=1}^N\int_{0}^t\left| z^N_i(s)-\hat{z}_i(s)\right|ds
+\int_{0}^t\sum_{i=1}^ND_{i}^3(s)ds\nonumber\\
&&\leq 2C\sum_{i=1}^N\int_{0}^t\hspace{-0.05in}\sup_{0\leq \tau \leq s}\left| z^N_i(s)-\hat{z}_i(s)\right|ds+\int_{0}^t\sum_{i=1}^ND_{i}^3(s)ds.
\label{eq:mvprf6}
\end{eqnarray}
We consider the last item in (\ref{eq:mvprf6}). We have that
\begin{eqnarray}
&&\hspace{-0.5in}\mathbb{E}\left| D_{i}^3(t)\right|^2\leq\int_{0}^t\mathbb{E}\bigg| \frac{1}{N}\sum_{j=1}^Nf\left(s,\hat{z}_{i}(s), \alpha\left(s,\varphi_{i}(s)\right),\hat{z}_{j}(s)\right)\nonumber\\
&&\hspace{0.5in}-\int_{\mathbb{R}}f\left[s,\hat{z}_i(s),\alpha\left(s,\varphi_i(s)\right),y\right]\mu_s(dy)\bigg|^2\label{eq:mvprf7}.
\end{eqnarray}
Define now $g(s,\hat{z}_i,x):=$ $f\left(s,\hat{z}_{i}(s), \alpha\left(s,\varphi_{i}(s)\right),x\right)-$ $f\left[s,\hat{z}_i(s),\alpha\left(s,\varphi_i(s)\right),\mu_s\right]$ and recall that $\varphi_i(t)$ depends on $\hat{z}_i(t)$ through (\ref{eq:mv-appr-obs-2}). Therefore, for $j\neq k$, we have
\begin{eqnarray}
\mathbb{E}\left[g(s,\hat{z}_i,\hat{z}_j)g(s,\hat{z}_i,\hat{z}_k)\right]=0, \label{eq:mvprf8}
\end{eqnarray}
which implies that there are no cross terms in (\ref{eq:mvprf7}). Consequently, by the boundedness of $f$ and the inequality that $\left(\sum_{i=1}^N x_i\right)^2$ $\leq N\sum_{i=1}^N x_i^2$, we obtain
\begin{eqnarray}
\mathbb{E}\left| D_{i}^3(t)\right|^2\leq k_1(t)/N=O(1/N)\label{eq:mvprf9}, 
\end{eqnarray}
where $k_1$ is an increasing function of $t$ but independent of $N$. Now by (\ref{eq:mvprf6}), (\ref{eq:mvprf9}) and Gronwall's lemma
\begin{eqnarray}
\sum_{i=1}^N\mathbb{E}\sup_{0\leq t \leq T}\left| z^N_i(t)-\hat{z}_i(t)\right |= O\left(\frac{1}{\sqrt{N}}\right),
\end{eqnarray}
which yields $\mathbb{E}\sup_{0\leq t \leq T}\left| z^N_i(t)-\hat{z}_i(t)\right |= O\left(\frac{1}{\sqrt{N}}\right)$. 
\end{proof}

\subsection{A Completely Observed Stochastic Optimal Control Problem for the Generic Agent} \label{sec:socp-gen}
The widely adopted procedure in the literature in the construction of a completely observed stochastic optimal control problem from the partially observed one is to use the unnormalized conditional density in the separation principle since it is known that the cost function under an equivalent measure is linear in the initial unnormalized conditional density. Furthermore, the dynamics of the unnormalized conditional density is also a linear functional of the initial density and hence one can significantly benefit from the unnormalized construction including the closed form computation of the first and second order functional (Fr\'{e}chet) derivatives with respect to the density-valued state component. However, in order to proceed with the unnormalized form, following the standard assumptions in the literature (see \cite{benes-karatsaz}, \cite{fleming81} and \cite{Bensoussan-book}), we shall restrict ourselves to the state dynamics in the following form:
\begin{itemize}
\item [(A6)] The function $f(t,x,u,y)$ is linear in the control: $f(t,x,u,y)=f^\dagger(t,x,y)+u$.
\end{itemize}
Recall that if the probability measure flow $\left(\mu_t\right)_{0\leq t \leq T}$ is fixed, $f[t,x,u,\mu]$ and $L[x,u,\mu]$ become a function of $(t,x,u)$ and as before, we denote 
\begin{eqnarray}
f^\ast(t,x,u):=f[t,x,u,\mu], ~ L^\ast(x,u):=L[x,u,\mu]\nonumber.
\end{eqnarray}
We need a further condition that the measure flow satisfies so that the induced functions are well behaved, see Definition 3 and Proposition 4 of \cite{hmc-mfg-main}. 
\begin{definition}
A probability measure flow $\mu_t$ on $[0,T]$ is in $\mathcal{M}_{[0,T]}$ if there exists $\beta \in (0,1]$ such that for any bounded and Lipschitz continuous function $\psi$ on $\mathbb{R}$,
\begin{eqnarray}
\sup_{1\leq j\leq K} \left|\int_{\mathbb{R}}\psi(y)\mu^j_{t'}(dy)-\int_{\mathbb{R}}\psi(y)\mu^j_{t''}(dy) \right|\leq B|t'-t''|^\beta,\label{eq:holder}
\end{eqnarray}
for all $t',t'' \in [0,T]$ where for given $\mu_t$, $B$ depends on upon the Lipschitz coefficient of $\psi$. 
\end{definition} 
In order to obtain the unnormalized filtering equations for the MV SDE, we first need to define an exponential martingale for the change of measure argument. Consider first the following MV SDE
\begin{eqnarray}
\hspace{-0.2in}dz^o(t)&=&f\left[t,z^o(t),\alpha(t),\mu_t\right]dt+ \sigma dw(t), \label{eq:mv-filt-eqn}\\
\hspace{-0.2in}dy(t)&=&h\left(t, z^o(t)\right)dt+ d\nu(t),~0 \leq t \leq T, \label{eq:mv-appr-obs-2a}
\end{eqnarray}
with $z^o(0)=z(0)$, $y(0)=0$ where $\left(\alpha(t)\right)_{0\leq t \leq T}$ is an admissible control, i.e., an $\mathcal{F}_t^y$-adapted process taking values in $U$. In the rest of this section, we assume that $\left(\mu_t\right) \in \mathcal{M}_{[0,T]}$ is fixed with exponent $\beta$ and we follow the approach presented in \cite{benes-karatsaz}. Hence, we define the process
\begin{eqnarray}
w^{-}(t)=w(t)-\int_{0}^tu(s)ds\label{eq:unnorm-bm}
\end{eqnarray}
and let introduce a new measure $\tilde{P}$ such that $\frac{dP}{d\tilde{P}}=M_{0}^{T}(u)$ where
\begin{eqnarray}
&&M_s^t(u):=\exp\biggl\{\int_{s}^t\left( u(\tau)dw^-(\tau)+h\left(\tau,z^o\left(\tau\right)\right)d\nu(\tau)\right)\nonumber\\
&&\hspace{0.5in}-\frac{1}{2}\int_{s}^t\left(|u(\tau)|^2+|h\left(\tau,z^o\left(\tau\right)\right)|^2\right)d\tau\biggr\}\label{ex:mart-unnorm}.
\end{eqnarray}
It now follows from Girsanov's theorem that under $\tilde{P}$, $y(t)$ is a Brownian motion. Define now the backward differential operator and its adjoint as follows: For $a \in U$
\begin{eqnarray}
\mathcal{J}_t^a&:=&\frac{1}{2}\partial^2_{xx}+(f^\dagger+a)\partial_x, \nonumber\\
\mathcal{J}_t^{\ast a}&:=& \frac{1}{2}\partial^2_{xx}-(f^\dagger+a)\partial_x-\partial_xf^\dagger\label{eq:fwd-bck-op}.
\end{eqnarray}
Similarly, for a given control process $u \in \mathcal{U}$, where $\mathcal{U}:=\bigl\{u(\cdot) \in U: u(t)~\text{is}~\mathcal{F}_t^{y}-\text{adapted}~\text{and}$ $~\mathbb{E}\int_{0}^T|u(t)|^2dt<\infty\bigr\}$, we denote the family of operators by
\begin{eqnarray}
\{\mathcal{J}_t^u:=\mathcal{J}_t^{u_t},~ \mathcal{J}_t^{\ast u}:=\mathcal{J}_t^{\ast u_t},~ 0\leq t\leq T\}\label{eq:fwd-bck-op-fam}.
\end{eqnarray}
Consider a random function $\{q(t,z;\tau,\kappa);\tau<t\leq T\}$ with $(z,\kappa) \in \mathbb{R}\times\mathbb{R}$ and assume that it is a fundamental solution of the Zakai equation (which is known to exist \cite{benes-karatsaz}) given by:
\begin{eqnarray}
&&\hspace{-0.19in}dq(t,z;\tau,\kappa)=\mathcal{J}_t^{\ast u}q(t,z;\tau,\kappa)dt+h(t,z)q(t,z;\tau,\kappa)dy(t), \nonumber\\
&&\hspace{-0.33in}\lim_{t\downarrow \tau }q(t,z;\tau,\kappa)=\delta_{z-\kappa},~\tau \leq t \leq T,~\tilde{P}-a.s.\label{eq:port-mck-fund-zakai}
\end{eqnarray}
Let $p(z)$ denote the density of $z(0)$ and set $q_t(z;\kappa):=q(t,z;0,\kappa)$. Then by \cite[Theorem 4.1]{benes-karatsaz} the function
\begin{eqnarray}
p_t(z)= \frac{\int_{\mathbb{R}}q_t(z;\kappa)p(\kappa)d\kappa}{\int_{\mathbb{R}}\int_{\mathbb{R}}q_t(z;\kappa)p(\kappa)d\kappa dz}, \label{eq:norm-dens-1}
\end{eqnarray}
is a version of the conditional density of $P\left(z^o(t) \in A|\mathcal{F}_t^{y}\right)$ i.e., for $\ell \in C_b\left(\mathbb{R}\right)$ and $A \in \mathcal{B}(\mathbb{R})$,
\begin{eqnarray}
 \mathbb{E}\left[\ell\left(z^o(T)\right)|\mathcal{F}_T^{y}\right]=\int_{\mathbb{R}}p_T(z)\ell(z)dz~~\tilde{P}-a.s.\label{eq:port-mck-cond-expec-1}
\end{eqnarray}
Finally, let us set $\tilde{\varphi}(t,z):=\int_{\mathbb{R}}q_t(z;\kappa)p(\kappa)d\kappa$. Then, by \cite[Theorem 4.1]{benes-karatsaz}, we obtain
\begin{eqnarray}
d\tilde{\varphi}(t,z)&=&\mathcal{J}_t^{\ast u}\tilde{\varphi}(t,z)dt+h(t,z)\tilde{\varphi}(t, z)dy(t),~ 0 \leq t \leq T, \nonumber\\
\tilde{\varphi}(0,z)&=&p(z), \label{eq:dens-filt-main}
\end{eqnarray}
where (\ref{eq:dens-filt-main}) is the Zakai equation for the unnormalized conditional density which will serve as the infinite dimensional state process of the completely observed optimal control problem. It is worthwhile recalling at this point that the goal is to solve the partially observed SOCP at the infinite population limit for which we aim to obtain an HJB equation in a function space by constructing the associated completely observed SOCP. We now proceed with such a derivation the first step of which requires one to define the cost in terms of the conditional density process and the new measure defined via (\ref{ex:mart-unnorm}). 

Indeed, consider the cost function and note that
\begin{eqnarray}
J(u;p)&=&\mathbb{E}\int_{0}^TL\left[z^o(t), u(t), \mu_t\right]dt \nonumber\\
&=&\tilde{\mathbb{E}}\int_{0}^T\left(\int_{\mathbb{R}}\int_{\mathbb{R}}L\left[z,u(t), \mu_t\right]q_t(z;x)p(x)dxdz\right)dt, \label{eq:equiv-cost-unnorm}
\end{eqnarray}
where $\tilde{\mathbb{E}}$ denotes expectation with respect to $\tilde{P}$ and (\ref{eq:equiv-cost-unnorm}) follows from \cite[Equation 5.1]{benes-karatsaz}. Note that we explicitly indicate dependence on the initial condition. Define the following space of functions:
\begin{eqnarray}
\mathsf{E}_k\bydef\left\{p\in L^1(\mathbb{R});\|p\|_k=\int_{\mathbb{R}}\left(1+|z|^k\right)|p(z)|dz<\infty\right\}\label{eq:norm-space}.
\end{eqnarray}
In the derivation of the HJB equation we consider the function space (\ref{eq:norm-space}) where for the expected total cost incurred during $[T-\tau, T]$, $0\leq \tau \leq T$, we assume that the initial condition satisfies $p_{T-\tau}(z)\in \mathsf{E}_k$ and hence, for a constant control $u_t=a \in U$ for all $0\leq t \leq T$, we have $J:[0,T]\times\mathsf{E}_k\rightarrow \mathbb{R}$. 
\begin{definition}\cite{benes-karatsaz}\label{def:spaceM}
Consider a probability space $\bigl(\Omega,\mathcal{F},\{\mathcal{F}_t\}_{t\geq 0},P\bigr)$ and an $\mathsf{E}_l$-valued stochastic process $\left(\eta(t,z)\right)_{0\leq t \leq T}$ adapted to the filtration $\mathcal{F}_t$ with $l \geq 0$. If
\begin{eqnarray}
\mathbb{E}\int_{0}^T \left(\int_{\mathbb{R}}\left(1+|z|^l\right)|\eta(t,z)|\right)^jdt<\infty,  \label{eq:norm-exp-space}
\end{eqnarray}
than we say that $\eta(t,z) \in \mathsf{M}_{l,j}[\mathcal{F}_t]$. 
\end{definition}
A continuous cost functional is next defined by setting: 
\begin{eqnarray}
V^u(\tau,p):=\mathbb{E}\int_{T-\tau}^TL\left[z^o(t), u(t), \mu_t\right]dt, \label{eq:val-func}
\end{eqnarray}
where $(\tau, p) \in [0,T]\times \mathsf{E}_k$ and $z^o(T-\tau)$ has a distribution with density $p$. We now recall the definition of the Fr\'{e}chet derivative. A function $f:\mathsf{X}\rightarrow \mathsf{Y}$ is said be Fr\'{e}chet differentiable at $x$ if there exists $\mathsf{D}f(x) \in \mathcal{L}(\mathsf{X};\mathsf{Y})$, where $\mathcal{L}(\mathsf{X};\mathsf{Y})$ denote the space of bounded linear operators from $\mathsf{X}$ to $\mathsf{Y}$, such that $\lim_{0\neq \| h \| \rightarrow 0} \frac{\| f(x+h)-f(h)-\mathsf{D}f(x)\cdot h\|_{\mathsf{Y}}}{\| h \|_{\mathsf{X}}}=0$. One can define higher order Fr\'{e}chet derivatives in a similar manner. For instance, the second order Fr\'{e}chet derivative of $f$ at $x \in \mathsf{X}$ satisfies that $\mathsf{D}^2f(x)\in \mathcal{L}\left(\mathsf{X};\mathcal{L}(\mathsf{X};\mathsf{Y})\right)$. We define the following assumptions. 
\begin{itemize}
\item [(A7)] The function $V(\tau,p):[0,T]\times \mathsf{E}_k\rightarrow \mathbb{R}$ possesses continuous first derivatives in $\tau$ and first and second order Fr\'{e}chet derivatives $\mathsf{D}V(\tau,p)$ and $\mathsf{D}^2V(\tau,p)$ with respect to $p$ in the form of linear functional and a bilinear form, respectively, which are given by
\begin{eqnarray}
\mathsf{D}V(\tau,p)[\eta]&=&\int_{\mathbb{R}}V_p\left(\tau,p\right)(z)\eta(z)dz, \nonumber\\
\mathsf{D}^2V(\tau,p)[\eta,\theta]&=&\int_{\mathbb{R}}\int_{\mathbb{R}}V^2_{pp}\left(\tau,p\right)(z,z')\eta(z)\theta(z')dzdz', ~\eta(\cdot), \theta(\cdot)  \in \mathsf{E}_l, \nonumber
\end{eqnarray}
where the kernels $V_p\left(\tau,p\right)(z)$ and $V^2_{pp}\left(\tau,p\right)(z,z')$ are continuous in their arguments and satisfy the following:
\begin{eqnarray}
|V_p\left(\tau,p\right)(z)|&\leq& \zeta_1\left(\tau,\|p \|_{l}\right)\left(1+|z|^l\right), \nonumber\\
|V^2_{pp}\left(\tau,p\right)(z,z')|&\leq& \zeta_2\left(\tau,\|p \|_{l}\right)\left(1+|z|^l\right)\left(1+|z'|^l\right), \label{eq:frechet-kernels}
\end{eqnarray}
for $\zeta_1, \zeta_2$ being continuous functions on $[0,T]\times \mathbb{R}_{+}$.
\item [(A8)] Consider (\ref{eq:dens-filt-main}). Assume that 
\begin{eqnarray}
\mathcal{J}_t^{\ast u}\tilde{\varphi}(z) & \in &  \mathsf{M}_{l,1}[\mathcal{F}_t^y]\cap\mathsf{M}_{l,2}[\mathcal{F}_t^y], \nonumber\\
h(t,z)\tilde{\varphi}(t, z)  & \in & \mathsf{M}_{l,2}[\mathcal{F}_t^y]\cap \mathsf{M}_{l,4}[\mathcal{F}_t^y], \label{eq:cond-zakai}
\end{eqnarray}
for some $l \geq 0$. 
\end{itemize}
Let $u_t=a$, $0\leq t \leq T$ and define $N_t(x):=\int_{\mathbb{R}}L[z, a, \mu_t]$ $q(T,z;t,x)dz$ and $s:=T-\tau$. Notice that with this notation, we have that $V^a(\tau,p)=\tilde{\mathbb{E}}\left(N_s,p\right)$ where we use the notational convention that $(\alpha,\beta) := \int_{\mathbb{R}}\alpha(z)\beta(z)dz$. Consider $V^a$ and note that due to the linearity in the infinite dimensional component, for a fixed control $a$, the Fr\'{e}chet derivatives satisfy 
\begin{eqnarray}
\mathsf{D}V^a(\tau,p)[\eta]&=& \int_{\mathbb{R}}V^a_p(\tau,p)(x)\eta(x)dx, ~ \eta(\cdot) \in \mathsf{E}_{l-1}, \nonumber
\end{eqnarray}
where $V^a_p(\tau,p)(x)=\tilde{\mathbb{E}}N_{s}(x)$. Similarly, 
\begin{eqnarray}
\mathsf{D}^2V^a(\tau,p)[\eta,\theta]&=& \int_{\mathbb{R}^2}V^a_{pp}(\tau,p)(x,x')\eta(x)\theta(x')dxdx', \nonumber
\end{eqnarray}
where $\eta(\cdot), \theta(\cdot)\in \mathsf{E}_{l-1}$ and $V^a_{pp}(\tau,p)(x,x')=0$. Therefore, the first set of conditions of (A7) are already satisfied when the unnormalized conditional density is considered. 

We are now in the position to provide a HJB equation that the function given in (\ref{eq:val-func}) satisfies. 
\begin{proposition}\label{pro:mortensen}
Consider the probability space $\bigl(\Omega,\mathcal{F},\{\mathcal{F}_t^y\}_{t\geq 0},P\bigr)$ and any admissible control process $\{u_t;0\leq t \leq T\} \in \mathcal{U}$ along with $\left(z^o(t),y(t),\nu(t),w(t)\right)_{0\leq t \leq T}$. Assume that (A1), (A2), (A3), (A5), (A6) and (A8) hold. If the following equation 
\begin{eqnarray}
\hspace{-0in}\frac{\partial V(\tau,p)}{\partial \tau}&=&\frac{1}{2}\mathsf{D}^2V(\tau,p)\cdot[h(T-\tau)p, h(T-\tau)p]\nonumber\\
&&+\min_{\theta \in U}\left\{\left(\mathcal{J}_{T-\tau}^\theta\mathsf{D}V(\tau,p)(\cdot),p\right)+\left(L[\cdot, \theta, \mu_{T-\tau}],p\right)\right\}, \nonumber\\
 V(0,p)&=&0, ~(\tau,p) \in [0,T]\times \mathsf{E}_k, \label{eq:prop-mortensen}
\end{eqnarray}
has a solution $V(\tau,p):=[0,T]\times \mathsf{E}_k \rightarrow \mathbb{R}$ which satisfies the assumptions defined in (A7), then $V(\tau,p)$ is a lower bound to the cost achieved under the control process $u$, i.e.,:
\begin{eqnarray}
V^u(\tau,p):=\mathbb{E}\int_{T-\tau}^TL[z^o(t),u(t),\mu_t]dt\geq  V(\tau,p), \label{eq:prop-mort-lb}
\end{eqnarray}
for any $(\tau,p) \in [0,T]\times\mathsf{E}_k$.
\end{proposition}
Under the assumptions (A1)-(A3), (A5)-(A8) and the condition that the measure flow $(\mu_t)_{0\leq t \leq T} \in \mathcal{M}_{[0,T]}$ is fixed, the proof of this proposition follow from \cite[Theorem 5.2]{benes-karatsaz}. 

Notice now that the PDE described in (\ref{eq:prop-mortensen}) is difficult to analyze (notice the existence of a function space in the domain of $V$); indeed, the solution to such an equation is not completely understood in the literature. Hence, in order to proceed with the analysis of the PO MFG system, we assume the following. 
\begin{itemize}
\item[]\hspace{-0.4in}(A9) The equation (\ref{eq:prop-mortensen}) has  a unique solution $V(t,p):[0,T]\times\mathsf{E}_k\rightarrow \mathbb{R}$ with $V(t,p) \in C^{1,2}_{t,p}\left([0,T]\times \mathsf{E}_k\right)$. 
\end{itemize}
Notice that due to this assumption, the best response control process can now be given in the following separated form:
\begin{eqnarray}
u^\ast &=&\{u^\ast(t,p)=a^\ast\left(T-t, p_t\right);0 \leq t \leq T\}, \nonumber\\
a^\ast\left(\tau, p\right)&=&\arg\min_{a \in U} \left\{\left(\mathcal{J}_{T-\tau}^a\mathsf{D}V(\tau,p)(\cdot),p\right)+\left(L\left[\cdot, a, \mu_{T-\tau}\right],p\right)\right\}, \label{eq:cont-sep}
\end{eqnarray}
if the Zakai equation (\ref{eq:dens-filt-main}) is strongly solvable for an $\mathcal{F}_t^y$-adapted random function $\tilde{\varphi}(t,z)$ with $u(t)=a^\ast\left(T-t, \frac{\tilde{\varphi}(t)}{\left(\tilde{\varphi}(t),1\right)}\right)$. 

Following standard procedures, e.g., see \cite[Theorem 4]{sen-caines-po-sicon}, one can also show that the value function $\bar{V}: [0,T]\times\mathsf{E}_k\rightarrow \mathbb{R}$ defined as $\bar{V}(t,p):=\inf_{u \in \mathcal{U}} V^u(t, p)$ is a solution to the HJB equation given in (\ref{eq:prop-mortensen}).  

To summarize: in order to obtain its optimal control,  a generic agent solves its partially observed control problem defined by (\ref{eq:mv-filt-eqn}), (\ref{eq:mv-appr-obs-2a}) and (\ref{eq:equiv-cost-unnorm}) and obtains the optimal control law in the feedback form given in (\ref{eq:cont-sep}).  

However, as in the completely observable and hence in the finite dimensional cases, the existence of feedback control policies is in general not sufficient for the validity of the fixed point argument. Therefore, we further assume the following. 
\begin{enumerate}
\item[]\hspace{-0.4in}(A10) For each $(\mu_t) \in \mathcal{M}_{[0,T]}$, $u^\ast\left(t,p|\left(\mu_t\right)_{0\leq t \leq T}\right)$ is continuous in $(t,p) \in [0,T]\times\mathsf{E}_k$ and Lipschitz continuous in $p \in \mathsf{E}_k$. 
\end{enumerate}
Before we proceed with the fixed point analysis we remark the following. 
\begin{remark}\label{rem:sob}
For the stochastic partial differential equation (SPDE) defined in (\ref{eq:dens-filt-main}), a solution via a Sobolev space characterization is considered in \cite{xyz} where the solution is defined on $H^1$ where $H^1:=\left\{f\in L^2\left(\mathbb{R}\right):\frac{\partial f}{\partial x} \in L^2\left(\mathbb{R}\right)\right\}$ with the norm $\| \partial k\|_{H^1}:=\{\int_{\mathbb{R}}\left(|k(x)|^2+|\frac{\partial k}{\partial x}|^2\right)dx\}^{1/2}$. Therefore, the infinite dimensional state process takes values in $H^1$ and one can obtain a similar expression for the HJB equation for which the stochastic calculus for Hilbert space-valued stochastic process can be used. 
\end{remark}
\section{Analysis of the Partially Observed MFG System}\label{sec:soln-fo-socp}
Following the widely used approach in MFG theory, it is now required to demonstrate that when the completely observed SOCP derived in Section \ref{sec:socp-gen} is considered and solved by each generic agent, the corresponding strategies should collectively replicate the aggregate behavior, which is the system mean field. This corresponds to the fixed point argument of MFG analysis which is also referred to as Nash Certainty Equivalence (NCE). For such an analysis, it suffices to prove that the MFG system has a unique solution which can be achieved by proving that starting with an exogenous measure, $\mu_{(\cdot)}^o$, the composition map below has a fixed point in the space of probability measures.
\begin{align*}
& \begin{array}[c]{ccccc}
 \mu^o_{(\cdot)}&\stackrel{\textrm{MV}}{\longrightarrow}& z^o(\cdot)&\stackrel{\textrm{NLF}}{\longrightarrow}& \tilde{\varphi}(\cdot)\\
\uparrow&&&&\downarrow\scriptstyle{\textrm{HJB}}\\
 z^o(\cdot)  &\stackrel{\textrm{MV}}{\longleftarrow}& u^o(\cdot,p)  &\stackrel{\textrm{BRC}}{\longleftarrow} &  V(\cdot,p) 
\end{array}
\end{align*}
We now introduce some preliminary material about the metrics on a space of probability measures which can be found in \cite{hmc-mfg-main} and \cite{sznitman}. Let $C\left([0,T];\mathbb{R}\right)$ be the space of continuous functions on $[0,T]$. For $x, y \in C\left([0,T];\mathbb{R}\right) $ define the norm $\| x-y\|:=\sup_{t \in [0,T]}|x(t)-y(t)|$. Then, $\left(C\left([0,T];\mathbb{R}\right), \|\cdot\|\right)$ is a Banach space. Consider also the metric $\rho(x,y)=\sup_{t \in [0,T]}|x(t)-y(t)| \wedge 1$; one can show that the metric space $\left(C\left([0,T];\mathbb{R}\right), \rho\right)$ is complete and separable (Polish). Let $C_\rho:=C\left([0,T];\mathbb{R}\right)$. On $\left(C\left([0,T];\mathbb{R}\right), \|\cdot\|\right)$ we define the $\sigma$-algebra $\mathcal{F}$ generated by the cylindrical sets of the form $\big\{x(\cdot) \in C_\rho: x(t_i) \in \mathcal{B}_i; t_i \in [0,T],$  $i=1,\ldots, l\big\}$ where each $\mathcal{B}_i \in \mathcal{B}\left(\mathbb{R}\right)$ and $l$ is a positive integer. Let $\mathcal{M}\left(C_\rho\right)$ denote the space of all probability measures $m$ on $\left(C_\rho,\mathcal{F}\right)$. $\mathcal{M}\left(C_\rho\times C_\rho\right)$  denotes the space of all probability measures on the product space. Define the canonical process $X$ with the sample space $C_\rho$; $i.e., X_t(\xi)=\xi_t, ~\xi \in C_\rho$.
\begin{definition}\label{def:was-metric}
For $m_1,m_2 \in\mathcal{M}\left(C_\rho\right)$, the Wasserstein metric is defined as follows:
\begin{eqnarray}
D_T(m_1,m_2)=\inf_{\Upsilon \in \Pi(m_1,m_2)}\int_{C_{\rho}\times C_{\rho}}\left(\sup_{s \leq T}
\left|X_s(\xi_1)-X_s(\xi_2)\right| \wedge 1\right)d\Upsilon(\xi_1,\xi_2)\label{eq:wass-metric}
\end{eqnarray}
 where 
\begin{eqnarray}
 \Pi(m_1,m_2)&:=& \big\{\Upsilon \subset \mathcal{M}(C_\rho \times C_\rho):\Upsilon(A\times C([0,T];\mathbb{R}))=m_1(A)~\mbox{and}\nonumber\\
&&\hspace{0.5in} \Upsilon(C([0,T];\mathbb{R})\times A)=m_2(A),~ A \in \mathcal{B}(C([0,T];\mathbb{R}))\big\} \nonumber.
 \end{eqnarray} 
\end{definition}
Note that the metric space $\left(\mathcal{M}\left(C_\rho\right), D_T\right)$ is also Polish.

We continue with the existence and uniqueness proof for MV SDEs in the partially observed setup which is based on a fixed point argument in the space $\mathcal{M}\left(C_\rho\right)$. Hence, for $1\leq i \leq N$, consider first the following SDEs:
\begin{eqnarray}
\hspace{-0.2in}dz_i^o(t)&=&f\left[t,z_i^o(t),\alpha\left(t,\tilde{\varphi}_i(t)\right),\mu_t
\right]dt+ \sigma dw_i(t),  \label{eq:mv-appr-1-a}\\
\hspace{-0.2in}dy_i(t)&=&h\left(t,z_i^o(t)\right)dt+ d\nu_i(t),~\enspace 0 \leq t \leq T,  \label{eq:mv-appr-obs-1-a}
\end{eqnarray}where $z_i^o(0)=z_i(0)$ and $\tilde{\varphi}_i$ is generated by the unnormalized nonlinear filter and $\alpha$ is an admissible control. Let $m \in \mathcal{M}\left(C_\rho\right)$ and observe that one can re-write (\ref{eq:mv-appr-1-a}) by defining the random process $\vartheta(t)$ on $[0,T]$ as follows:
\begin{eqnarray}
&&\vartheta_i(t)=\int_{0}^t\int_{C_\rho}f\left(s,\vartheta_i(s),\alpha\left(s,\tilde{\varphi}_i(s)\right), \xi_s\right)dm(\xi)ds\nonumber\\
&&\hspace{0.5in}+z_i(0)+\int_{0}^t\sigma dw_i(s), ~~0\leq t \leq T.
\label{eq:def-fixed}
\end{eqnarray}
Let us denote the law of $\vartheta_i$ by $\Phi(m) \in \mathcal{M}\left(C_\rho\right)$.

Although the results which we derived in the previous sections hold with time varying observation dynamics, it is simpler to handle the sensitivity analysis of the filtering equation when the observation dynamics are time invariant. We therefore assume the following in the rest of the paper. 
\begin{enumerate}
\item[]\hspace{-0.4in}(A11) The observation dynamics is time invariant: $h(t,x)=h(x)$.
\end{enumerate}
\begin{theorem}\label{the:mv-sde-uniq}
Under (A0)-(A3), (A5) and (A10), there exists a unique consistent solution pair $\bigl(z_i^o(t),\mu_t\bigr)$ with $\mu_t \in \mathcal{M}_{[0,T]}$.
\end{theorem}
\begin{proof}
The proof is a generalization of \cite[Theorem 6]{hmc-mfg-main}, \cite[Theorem 6.12]{NP-Siam2013} and \cite[Theorem 13]{sen-caines-po-sicon} requires to consider an unnormalized conditional density in the control law. For $m,\hat{m} \in \mathcal{M}\left(C_\rho\right)$ let $\vartheta_i(t)$ and $\hat{\vartheta}_i(t)$ be defined by (\ref{eq:def-fixed}) corresponding to $m$ and $\hat{m}$, respectively, with the same initial condition $z_i(0)$. Similarly, let $\tilde{\varphi}(t)$ and $\hat{\tilde{\varphi}}(t)$ be generated by the unnormalized filtering equations for $\vartheta_i(t)$ and $\hat{\vartheta}_i(t)$, respectively. It follows that
\begin{eqnarray}
&&\hspace{0.5in}\sup_{0\leq s \leq t}\left| \vartheta_i(s)- \hat{\vartheta_i}(s)\right|\leq \label{eq:prf-mv-sde-1}\\
&&\int_{0}^t\bigg | \int_{C_\rho} f\left(s, \vartheta_i(s),\alpha\left(s,\tilde{\varphi}(s)\right),\xi_s\right)dm(\xi)-\int_{C_\rho} f\left(s, \hat{\vartheta}_i(s),\alpha\left(s,\hat{\tilde{\varphi}}(s)\right),\xi_s\right)d\hat{m}(\xi)\bigg | ds \nonumber.
\end{eqnarray}
For any $\bar{m} \in \mathcal{M}\left(C_\rho \times C_\rho\right)$ with marginals $(m,\hat{m})$, we have
\begin{eqnarray}
\hspace{-0in}\Lambda_s &=& \bigg | \int_{C_\rho} f\left(s, \vartheta_i(s),\alpha\left(s,\tilde{\varphi}(s)\right),\xi_s\right)dm(\xi)- \int_{C_\rho} f\left(s,\hat{\vartheta}_i(s),\alpha\left(s,\hat{\tilde{\varphi}}(s)\right),\xi_s\right)d\hat{m}(\xi)\bigg|\nonumber\\
&=&\bigg | \int_{C_\rho \times C_\rho} f\left(s,\vartheta_i(s),\alpha\left(s,\tilde{\varphi}(s)\right),\xi_s\right)d\bar{m}(\xi,\hat{\xi})\nonumber\\
&&\hspace{0.3in}-\int_{C_\rho\times C_\rho}f\left(s, \hat{\vartheta}_i(s),\alpha\left(s,\hat{\tilde{\varphi}}(s)\right),\hat{\xi}_s\right)d\bar{m}(\xi,\hat{\xi}) \bigg|\label{eq:prf-mv-sde-2}\\
&\leq& C_1\big(\big| \vartheta_i(s) - \hat{\vartheta}_i(s)\big|\wedge 1\big)+C_2\big(\big\| \tilde{\varphi}_i(s) - \hat{\tilde{\varphi}}_i(s)\big\|_{\mathsf{E}_k}\wedge 1\big)\nonumber\\
&&\hspace{0.3in}+\int_{C_\rho\times C_\rho}C_3\big(\big |\xi_s - \hat{\xi}_s\big| \wedge 1\big)d\bar{m}(\xi,\hat{\xi}) \label{eq:prf-mv-sde-3}
\end{eqnarray}
where (\ref{eq:prf-mv-sde-3}) follows due to the boundedness and the Lipschitz continuity of $f$ and $\alpha$. We note that the essential difference with the completely observed MV SDEs is the existence of the conditional density terms, that is to say the solutions to the nonlinear filtering equations in the Zakai form where the observation process $y(t)$ acts as the exogenous input process, which are going to be handled through the robust representation of the filtering processes.

Recall that for $0 \leq t \leq T$,
\begin{eqnarray}
dy(t)&=&h\bigl(\vartheta_i(t)\bigr)dt+ d\nu(t), \nonumber\\
d\hat{y}(t)&=&h\bigl(\hat{\vartheta_i}(t)\bigr)dt+ d\nu(t), \label{eq:prf-mv-sde-4}
\end{eqnarray}
and hence, the filtering processes $\tilde{\varphi}_i(t)$ and $\hat{\tilde{\varphi}}_i(t)$ are $\mathcal{F}_t^{y}$ and $\mathcal{F}_t^{\hat{y}}$-adapted, respectively. Let $\ell \in C_b(\mathbb{R})$ and for $0\leq t \leq T$ consider the following unnormalized conditional expectations:
\begin{eqnarray}
\tilde{\mathbb{E}}\big[\ell\left(\vartheta_i(t)\right)M_0^t(\alpha)|\mathcal{F}_t^y\big],~\tilde{\mathbb{E}}\big[\ell\big(\hat{\vartheta}_i(t)\big)M_0^t(\alpha)|\mathcal{F}_t^{\hat{y}}\big]\label{eq:prf-mv-sde-5}.
\end{eqnarray}
Let us also define the path valued random variable $y_{\cdot}:\Omega \rightarrow \left([0,t];\mathbb{R}\right)$ such that $y_{\cdot}\left(\omega\right)=\left(y(s,\omega), 0\leq s\leq t\right)$. Hence, by \cite[Theorem 5.12]{bain-crisan}, there exists a function $\eta^{\ell}: C\left([0,t];\mathbb{R}\right)
$ $\rightarrow \mathbb{R}$ such that
\begin{eqnarray}
&&\tilde{\mathbb{E}}\left[\ell\left(\vartheta_i(t)\right)M_0^t(\alpha)|\mathcal{F}_t^y\right]=\eta^{\ell}(y_{\cdot}), \nonumber\\
&&\tilde{\mathbb{E}}\big[\ell\big(\hat{\vartheta}_i(t)\big)M_0^t(\alpha)|\mathcal{F}_t^{\hat{y}}\big]=\eta^{\ell}(\hat{y}_{\cdot}),~\tilde{P}-a.s\label{eq:prf-mv-sde-5-a}.
\end{eqnarray}
Furthermore, the function $\eta^{\ell}$ is locally Lipschitz in the sup-norm and locally bounded \cite[Lemma 5.6]{bain-crisan}. To continue, recall that by \cite[Theorem 4.1]{benes-karatsaz}, 
\begin{eqnarray}
\tilde{\mathbb{E}}\left[\ell\left(\vartheta_i(t)\right)M_0^t(\alpha)|\mathcal{F}_t^y\right]
=\int_{\mathbb{R}}\ell(x)\tilde{\varphi}(x)dx,~\tilde{P}-a.s\label{eq:prf-mv-sde-6}.
\end{eqnarray} 
Take $l=(1+|x|^k)$ for a $k$ that is set in (\ref{eq:norm-space}) so that we have
\begin{eqnarray}
\left | \tilde{\mathbb{E}}\big[\ell\left(\vartheta_i(t)\right)M_0^t(\alpha)|\mathcal{F}_t^y\big]-\tilde{\mathbb{E}}\big[\ell\big(\hat{\vartheta}_i(t)\big)M_0^t(\alpha)|\mathcal{F}_t^{\hat{y}}\big]\right|&=&\left| \int_{\mathbb{R}}\left(1+\left| x\right|^k\right)\left(\tilde{\varphi}(x)-\hat{\tilde{\varphi}}(x)\right)dx\right|\nonumber\\
&=&\| \tilde{\varphi}(x)-\hat{\tilde{\varphi}}(x)\|_{\mathsf{E}_k}\label{eq:prf-mv-sde-7}.
\end{eqnarray}
Notice that $\eta^{(1+|x|^k)}$ is only locally Lipschitz, however, since $y_{\cdot}$ and $\hat{y}_{\cdot}$ take values in $C\left([0,t];\mathbb{R}\right)$, there exists $R' >0$ such that $\|y_{\cdot}\| \leq R'$ for all $\omega \in \Omega$. Hence let $R>R'$ and consequently for any $\|y_{\cdot}\|, \|\hat{y}_{\cdot}\| \leq R$, there exists a constant $C_f^R>0$ such that
\begin{eqnarray}
\| \tilde{\varphi}(x)-\hat{\tilde{\varphi}}(x)\|_{\mathsf{E}_k}&=&\| \eta^{(1+|x|^k)}\left(y_{\cdot}\right)-\eta^{(1+|x|^k)}\left(\hat{y}_{\cdot}\right)\| \nonumber\\
&\leq& C^{R}_{f}\sup_{0\leq s \leq t}\left|y(s) -\hat{y}(s) \right| \leq C^{R}_{f}C_4\sup_{0\leq s \leq t}\left|\vartheta_i(s) -\hat{\vartheta_i}(s) \right|\label{eq:prf-mv-sde-8}
\end{eqnarray}
where $C_4$ is a constant obtained from the Lipschitz continuity of $h$. Substituting (\ref{eq:prf-mv-sde-8}) in (\ref{eq:prf-mv-sde-3}) yields 
\begin{eqnarray}
\Lambda_s \leq \left(C_1+C_2C^{R}_{f}C_4\right)\big(\big| \vartheta_i(s) - \hat{\vartheta}_i(s)\big|\wedge 1\big)+\int_{C_\rho\times C_\rho}C_3\big(\big |\xi_s - \hat{\xi}_s\big| \wedge 1\big)d\bar{m}(\xi,\hat{\xi})\label{eq:prf-mv-sde-9}. 
\end{eqnarray}
Clearly, by (\ref{eq:prf-mv-sde-9}) the proof easily follows from \cite[Theorem 6]{hmc-mfg-main}. We herein provide details for the sake of completeness. Notice first that (\ref{eq:prf-mv-sde-9}) implies
\begin{eqnarray}
\Lambda_s \leq\left(C_1+C_2C^{R}_{f}C_4\right)(| \vartheta_i(s) - \hat{\vartheta}_i(s)|\wedge 1)+C_3D_s(m,\hat{m}), \label{eq:prf-mv-sde-10}
\end{eqnarray}
since $\bar{m}$ is any measure with marginals $m$ and $\hat{m}$. (\ref{eq:prf-mv-sde-3}) and (\ref{eq:prf-mv-sde-10}) together imply
\begin{eqnarray}
&&\hspace{-0.5in}\sup_{0\leq s \leq t}\left |\vartheta_i(s)- \hat{\vartheta}_i(s)\right | \leq \nonumber\\
&&\int_{0}^t\biggl[C_3D_s(m,\hat{m})+\left(C_1+C_2C^{R}_{f}C_4\right)\left(\left| \vartheta_i(s) - \hat{\vartheta}_i(s)\right |\wedge 1\right)\biggr]ds\label{eq:prf-mv-sde-11}.
\end{eqnarray}
Applying Gronwall's lemma yields
\begin{eqnarray}
\sup_{0\leq s \leq t}\left |\vartheta_i(s)- \hat{\vartheta}_i(s)\right | \wedge 1 \leq C_T^R\int_{0}^tC_3D_s(m,\hat{m})ds\label{eq:prf-mv-sde-12}.
\end{eqnarray}
where $C_T^R=\exp\left(\left(C_1+C_2C^{R}_{f}C_4\right)T\right)$. Notice that $\vartheta_i$ and $\hat{\vartheta}_i$ induce two probability distributions, denoted by $\Phi(m)$ and $\Phi(\hat{m})$, respectively, on $C_{\rho}$, Likewise, the joint distribution of $(\vartheta_i,\hat{\vartheta}_i)$ gives a measure $\bar{m}_{\Phi}$ on $C_{\rho} \times C_{\rho}$. Taking expectation of both sides of (\ref{eq:prf-mv-sde-13}) we obtain
\begin{eqnarray}
 D_t\big(\Phi(m),\Phi(\hat{m})\big)\leq C_T^R\int_{0}^tC_3D_s(m,\hat{m})ds\label{eq:prf-mv-sde-13}.
\end{eqnarray}
The proof of the existence and uniqueness is now complete by following the Cauchy argument in \cite{sznitman}. Finally, the claim that $\mu_t \in \mathcal{M}_{[0,T]}$ directly follows from \cite[Lemma 7]{hmc-mfg-main}. 
\end{proof}		
By the analysis in Section \ref{sec:socp-gen}, we have first obtained the optimal control law for the generic agent by assuming that (i) it has only access to partial information on its own state and (ii) the flow of probability measures is fixed. In the second step we proved that the closed loop MV dynamics of the generic agent has a unique solution when the agent uses the Lipschitz control strategy obtained as the solution of HJB equation derived in Section \ref{sec:socp-gen}. Furthermore, it follows that after all agents apply such Lipschitz control strategies, the resulting measure flow maintains a certain continuity. Hence, one can refine its strategy by solving the HJB equation of the partially observed control problem by using this new measure. In the next section we discuss this in more detail.
\subsection{Fixed Point Analysis and the Main Result}\label{sec:fixedpoint}
Given $\mu_t^o \in \mathcal{M}_{[0,T]}$, by the proposition \ref{pro:mortensen}, we obtain a solution for $V(\cdot)$ and subsequently, corresponding to each $(t,p) \in [0,T]\times\mathsf{E}_k$ we get $u(t,p)$ as a well defined function minimizing (\ref{eq:prop-mortensen}). Consequently, we write the optimal control law in the feedback form
\begin{eqnarray}
u=u^\ast\left(t,p\big|\left(\mu^o_t\right)_{0\leq t \leq T}\right), ~ (t, p) \in [0,T]\times\mathsf{E}_k, \label{eq:fixedpoint-1}
\end{eqnarray}
for which we define the well defined map: $\Upsilon:\mathcal{M}_{[0,T]} \rightarrow C_{\rm{Lip}(p)}\left([0,T]\times\mathsf{E}_k;U\right)$ with
\begin{eqnarray}
\Upsilon\left(\left(\mu^o_t\right)_{0\leq t \leq T}\right):=u^\ast\left(t,p\big |\left(\mu^o_t\right)_{0\leq t \leq T}\right)\label{eq:map-meas-cont}.
\end{eqnarray}
Notice that the map characterizes the interaction of individual and the mass which can be considered as the best response map; the individual optimal decision is obtained while the actions of all other agents are fixed. 

We now consider the second component of the fixed point argument. Given a function $\alpha(t,p) \in C_{\rm{Lip}(p)}\left([0,T]\times\mathsf{E}_k;U\right)$ we implement it as a control law in (\ref{eq:mv-appr-1-a}) which leads to a well defined solution $\left(z^o(t),y(t)\right)$, $0\leq t\leq T$. Let us denote the law of the solution of $z^o$ by $m \in \mathcal{M}\left(C_{\rho}\right)$. Then, we define the map $\hat{\Upsilon}: C_{\rm{Lip}(p)}\left([0,T]\times\mathsf{E}_k;U\right) \rightarrow \mathcal{M}\left(C_{\rho}\right)$ by
\begin{eqnarray}
m=\hat{\Upsilon}\left(\alpha\right)\label{eq:map-cont-meas}.
\end{eqnarray}
Note that by Theorem \ref{the:mv-sde-uniq}, $\left(\mu_t\right)_{0\leq t \leq T}$, the law of $z^o(t)$, is in $\mathcal{M}_{[0,T]}$ and hence one can also specify the well defined map $\bar{\Upsilon}:  C_{\rm{Lip}(p)}\left([0,T]\times\mathsf{E}_k;U\right) \rightarrow \mathcal{M}_{[0,T]}$ by
\begin{eqnarray}
\left(\mu_t\right)_{0\leq t \leq T}=\bar{\Upsilon}\left(\alpha\right)\label{eq:map-cont-dist}.
\end{eqnarray}
Therefore, we obtain the following proposition as the partially observed equivalent of \cite[Proposition 8]{hmc-mfg-main}.
\begin{proposition}
Assume (A0)-(A10) and $\left(\mu^o_t\right)_{0\leq t \leq T} \in \mathcal{M}_{[0,T]}$. We have $\bar{\Upsilon}\circ\Upsilon\left(\left(\mu^o_t\right)_{0\leq t \leq T}\right) \in \mathcal{M}_{[0,T]}$, i.e., $\Upsilon_{\mathcal{M}}:=\bar{\Upsilon}\circ\Upsilon:  \mathcal{M}_{[0,T]} \rightarrow  \mathcal{M}_{[0,T]}$.
\end{proposition}
It is now clear that by the construction of $\Upsilon$ and $\bar{\Upsilon}$ we obtain a solution to NCE system if we can find $\mu_t^o \in \mathcal{M}_{[0,T]}$ that satisfies the fixed point equation
\begin{eqnarray}
\bar{\Upsilon}\circ\Upsilon\left(\left(\mu^o_t\right)_{0\leq t \leq T}\right)=\left(\mu^o_t\right)_{0\leq t \leq T}\label{eq:fixedpoint}.
\end{eqnarray}
As mentioned in \cite{hmc-mfg-main}, there exists several difficulties in demonstrating the existence of a fixed point for (\ref{eq:fixedpoint}); most notably, the sensitivity of the optimal control policies with respect to the measure flow $\left(\mu^o_t\right)_{0\leq t \leq T}$ and hence, such a regularity condition is taken as an assumption in \cite[See (37)]{hmc-mfg-main}. By generalizing the approach presented in \cite{NP-Siam2013}, we derive a similar sensitivity analysis for the partially observed setup. 

Let $0\leq t \leq T$ and $s=T-t$ and consider the best response process given by the solution of HJB equation in (\ref{eq:prop-mortensen}):
\begin{eqnarray}
u^\ast(t,p_t)&=&\arg\min_{a \in U} \left\{\left(\mathcal{J}_{t}^a\mathsf{D}V(s,p_t)(\cdot),p_t\right)+\left(L[\cdot, a, \mu_{t}^o],p_t\right)\right\}\nonumber\\
&=&\arg\min_{a \in U} \left\{a\left(\int\partial_{x}V_p(s,p_t)(x)p_t(x)dx+L[x, a, \mu_{t}^o]p_t(x)dx\right)\right\}\label{eq:cont-sep-fixed}.
\end{eqnarray}
Define the Hamiltonian 
\begin{eqnarray}
H\left(t,p,a,\partial_{x}V_p(s,p)(x)\right):=a\left(\int_{\mathbb{R}}\partial_{x}V_p(s,p)(x)p(x)dx+L[x, a, \mu_t^o]p(x)dx\right)\label{eq:hamiltonian}.
\end{eqnarray}
We also define the closed loop dynamics by
\begin{eqnarray}
dz^o(t)&=&f\left[t,z^o(t),u(t),\mu_t^o
\right]dt+ \sigma dw(t),z^o(0)=z(0),\nonumber\\
dy(t)&=&h\left(z^o(t)\right)dt+ d\nu(t), ~y(0)=0 .\label{eq:cl-loop-dyn-fp}
\end{eqnarray}
We  assume the following. 
\begin{enumerate}
\item[]\hspace{-0.4in}(A11) For each $\mu_t \in \mathcal{M}_{[0,T]}$, the set
\begin{eqnarray}
S(t,p,q):=\argmin_{a \in U}H\left(t,p,a,q\right), \label{eq:as-single}
\end{eqnarray}
is singleton and the resulting $u^\ast$ as a function of $(t,p,q) \in [0,T]\times\mathsf{E}_k \times \mathbb{R}$ is continuous in $t$, Lipschitz continuous in $(p,q)$, uniformly with respect to $t$ and $\mu_t \in \mathcal{M}_{[0,T]}$.
\end{enumerate}
The conditions under which the above assumptions hold in the partially observed case are beyond the scope of this work. However, in the completely observed situation, such conditions can be satisfied under sutiable convexity assumptions in the control variable; the reader is referred to \cite{hmc-mfg-main} for more details.

Following \cite{kol-nonlinearmfg}, we define the G\^ateux derivative of a function $F(t, p,\mu): [0,T]\times \mathsf{E}_k \times \mathcal{P}\left(\mathbb{R}\right)\rightarrow \mathbb{R}$ with respect to the measure $\mu(y)$ as follows :
\begin{eqnarray}
\partial_{\mu(y)}F(t,p,\mu)=\lim_{\epsilon \rightarrow 0}\frac{F\left(t,p,\mu+\epsilon\delta(y)\right)-F(t,p,\mu)}{\epsilon}, \label{eq:def-gateux}
\end{eqnarray} 
where $\delta$ is the Dirac delta function. We assume the following. 
\begin{enumerate}
\item[]\hspace{-0.4in}(A12) The G\^ateux derivatives of $f$ and $L$ with respect to $\mu$ exists are $C^{\infty}(\mathbb{R})$ and uniformly bounded. Let $V(t,p):[0,T]\times \mathsf{E}_k\rightarrow \mathbb{R}$ be the unique solution to the HJB equation in (\ref{eq:prop-mortensen}). Then, $V(t,p)$ is uniformly bounded, its G\^ateaux derivative $\mathsf{D}V_p(t,p)$ (or the Kernel $V_p(t,p)(x)$) exists and is uniformly bounded with respect to all its parameters. 
\end{enumerate}
\begin{theorem}\label{the:sens-cont}
Assume (A0)-(A12) hold and in addition assume that the resulting control law is Lipschitz in $\mu$. Then for given $\left(\mu_t\right)_{0\leq t \leq T}, \left(\tilde{\mu}_t\right)_{0\leq t \leq T} \in \mathcal{M}_{[0,T]}$, there exists a constant $c_1$ such that
\begin{eqnarray}
\sup_{t,p \in [0,T]\times\mathsf{E}_k} |u^\ast(t,p) - \tilde{u}^\ast(t,p)| \leq c_1 \bigl(D_T(\mu, \tilde{\mu})\bigr)\label{eq:theo-sens-cont}.
\end{eqnarray}
\end{theorem}
\begin{proof}
Note that Assumption (A12) and the fact that resulting $u^\ast$ is Lipschitz continuous in $\mu$ yields
\begin{eqnarray}
|u^\ast(t,p) - \tilde{u}^\ast(t,p)| &\leq& k_1 D_t\bigl(\mu, \tilde{\mu}\bigr)+k_2 |\partial_{x}V^\mu_p(t,p)(x) - \partial_{x}V^{\tilde{\mu}}_p(t,p)(x)|, \label{eq:prf-sens-cont-1}
\end{eqnarray}
with positive constants $k_1$ and $k_2$ where $V^\mu_p(t,p)(x)$ and $V^{\hat{\mu}}_p(t,p)(x)$ are the kernels defined in (\ref{eq:frechet-kernels}) corresponding to measures $\mu$ and $\hat{\mu}$.

The goal is to use the existence of G\^ateux derivative of the Kernel $\partial_xV_p(t,p)(x)$ with respect to the measure $\mu$ which satisfy the boundedness assumption so that one can invoke the mean value theorem (MVT). Indeed, by the assumption that $\partial_{\mu}\partial_xV_p(t,p)(x)$ is uniformly bounded, by MVT we obtain
\begin{eqnarray}
|\partial_{x}V^\mu_p(t,p)(x) - \partial_{x}V^{\tilde{\mu}}_p(t,p)(x)|\leq k_3D_t\bigl(\mu, \tilde{\mu}\bigr)\label{eq:prf-sens-cont-2}.
\end{eqnarray} 
Consequently, we obtain
\begin{eqnarray}
|u^\ast(t,p) - \tilde{u}^\ast(t,p)| &\leq& (k_1+ k_2k_3) D_t\bigl(\mu, \tilde{\mu}\bigr)|\label{eq:prf-sens-cont-3},
\end{eqnarray}
which completes the proof with $c_1=k_1+ k_2k_3$.
\end{proof}
\begin{remark}\label{rem:sens-1}
In Theorem \ref{the:sens-cont} above, we assume the uniform boundedness of the G\^ateux derivative of the function that satisfies the HJB equation obtained for the partially observed setup. We note that one could follow a similar approach to \cite[Section 6]{kol-nonlinearmfg} in order to analyze the boundedness property of the solution of the HJB equation (\ref{eq:prop-mortensen}) (see also \cite[Appendix F]{NP-Siam2013} for the finite dimensional case) and its Fr\'{e}chet derivatives by analyzing associated kernels. However, such an analysis would considerably depend upon the analysis of PDE in the form (\ref{eq:prop-mortensen}) which is not well understood in the literature.
\end{remark}
We now provide a sensitivity result for the distance between two measures with respect to the control policies by combining the general developed in \cite{hmc-mfg-main}, \cite{NP-Siam2013} and \cite{sen-caines-po-sicon}. 
\begin{lemma}\label{lem:sens-meas-con}
Under (A0)-(A11) there exists a constant $c_2$ such that 
\begin{eqnarray}
D_T\left(m,\hat{m}\right)\leq c_2 \sup_{t, p \in [0,T]\times \mathsf{E}_k}\left| u^\ast(t,p)-\hat{u}^\ast(t,p)\right|, \label{eq:lem-sens-meas-cont}
\end{eqnarray}
where $m,\hat{m} \in \mathcal{M}(C_{\rho})$ are induced by (\ref{eq:map-cont-meas}) using $u^\ast, \hat{u}^\ast \in C_{\rm{Lip}(p)}\left([0,T]\times\mathsf{E}_k;U\right)$.
\end{lemma} 
\begin{proof}
Denote the two solutions corresponding to $u^\ast$ and $\tilde{u}^\ast$ by $z^o(t)$ and $\hat{z}^o(t)$. Hence,
\begin{eqnarray}
&&z^o(t)=\int_{0}^t\int_{C_\rho}f\left(s,z^o(s),u^\ast\left(s,\tilde{\varphi}(s)\right), \xi_s\right)dm(\xi)ds+z(0)+\int_{0}^t\sigma dw(s), \nonumber\\
&&\hat{z}^o(t)=\int_{0}^t\int_{C_\rho}f\left(s,\hat{z}^o(s),\hat{u}^\ast\left(s,\hat{\tilde{\varphi}}(s)\right), \hat{\xi}_s\right)d\hat{m}(\hat{\xi})ds + z(0)+\int_{0}^t\sigma dw(s)\label{eq:prf-lem-sensmeascont1}.
\end{eqnarray}
By the Lipschitz continuity of $f$, $u^\ast$ and $\hat{u}^\ast$ we obtain
\begin{eqnarray}
&&\left| f\left(s, z^o(s),u^\ast\left(s,\tilde{\varphi}(s)\right), \xi_s\right) - f\bigl(s,\hat{z}^o(s),\hat{u}^\ast\bigl(s,\hat{\tilde{\varphi}}(s)\bigr), \hat{\xi}_s\bigr)\right|\nonumber\\
&\leq&\big| f\left(s,z^o(s),u^\ast\left(s,\tilde{\varphi}(s)\right), \xi_s\right) - f\bigl(s,\hat{z}^o(s),u^\ast\bigl(s,\hat{\tilde{\varphi}}(s)\bigr), \hat{\xi}_s\bigr)\big|\nonumber\\
&&\hspace{0.1in}+\big| f\big(s,\hat{z}^o(s),u^\ast\big(s,\hat{\tilde{\varphi}}(s)\big), \hat{\xi}_s\big) -f\bigl(s,\hat{z}^o(s),\hat{u}^\ast\bigl(s,\hat{\tilde{\varphi}}(s)\bigr), \hat{\xi}_s\bigr)\big|\nonumber\\
&\leq & C_1\left(\left|z^o(s)-\hat{z}^o(s)\right|\wedge 1\right)+C_2\bigl(\| \tilde{\varphi}(s)-\hat{\tilde{\varphi}}(s)\|_{\mathsf{E}_k}\wedge 1\bigr)\nonumber\\
&&\hspace{0.1in}+C_3\left(\left|\xi_s-\hat{\xi}_s\right|\wedge 1\right)+C_5\sup_{(s,p)\in [0,T]\times\mathsf{E}_k}\left|u^\ast(s,p)-\hat{u}^\ast(s,p)\right|.
\label{eq:prf-lem-sensmeascont2}
\end{eqnarray}
Following similar steps in the proof of Theorem \ref{the:mv-sde-uniq} we get 
\begin{eqnarray}
\left|z^o(t)-\hat{z}^o(t)\right|&\leq &\int_{0}^t C_1\left(\left|z^o(s)-\hat{z}^o(s)\right|\wedge 1\right)ds+\int_{0}^t C_2\bigl(\| \tilde{\varphi}(s)-\hat{\tilde{\varphi}}(s)\|_{\mathsf{E}_k}\wedge 1\bigr)ds\nonumber\\
&&+C_3\int_{0}^tD_s(m,\hat{m})ds+C_5t\sup_{(s,p)\in [0,T]\times\mathsf{E}_k}\left|u^\ast(s,p)-\hat{u}^\ast(s,p)\right|.\label{eq:prf-lem-sensmeascont3}
\end{eqnarray}
Now by use of the robust representation of nonlinear filter demonstrated in (\ref{eq:prf-mv-sde-4})-(\ref{eq:prf-mv-sde-9}), we obtain
\begin{eqnarray}
&&\hspace{-0.2in}\left|z^o(t)-\hat{z}^o(t)\right|\nonumber\\
&&\leq \int_{0}^t C_1\left(\left|z^o(s)-\hat{z}^o(s)\right|\wedge 1\right)+\left(C_1+C_2C^{R}_{f}C_4\right)\sup_{0\leq s \leq t}\bigl(|z^o(s)-\hat{z}^o(s)|\wedge 1\bigr)ds\nonumber\\
&&+C_3\int_{0}^tD_s(m,\hat{m})ds+C_5t\sup_{(s,p)\in [0,T]\times\mathsf{E}_k}\left|u^\ast(s,p)-\hat{u}^\ast(s,p)\right|\label{eq:prf-lem-sensmeascont4}
\end{eqnarray}
where $C_f^R$ and $C_4$ are defined in (\ref{eq:prf-mv-sde-8}). Applying Gronwall's lemma to (\ref{eq:prf-lem-sensmeascont4}) gives
\begin{eqnarray}
&&\sup_{0\leq s \leq t}\left|z^o(t)-\hat{z}^o(t)\right| \wedge 1\leq \nonumber\\
&&\exp\left(\left(C_1+C_2C^{R}_{f}C_4\right)T\right)\bigg(C_3\int_{0}^t\hspace{-0.02in}D_s(m,\hat{m})ds+C_5t\sup_{(s,p)\in [0,T]\times\mathsf{E}_k}\hspace{-0.04in}\left|u^\ast(s,p)-\hat{u}^\ast(s,p)\right|\bigg)\nonumber.
\end{eqnarray}
Consequently, 
\begin{eqnarray}
&&\hspace{0.3in}D_t(m,\hat{m})\leq \label{eq:prf-lem-sensmeascont5}\\
&&\exp\left(\left(C_1+C_2C^{R}_{f}C_4\right)T\right)\bigg(C_3\int_{0}^tD_s(m,\hat{m})ds+C_5t\sup_{(s,p)\in [0,T]\times\mathsf{E}_k}\left|u^\ast(s,p)-\hat{u}^\ast(s,p)\right|\bigg)\nonumber.
\end{eqnarray}
Applying Gronwall's lemma to (\ref{eq:prf-lem-sensmeascont5}) we complete the proof.
\end{proof}
We can now present the main result of PO MFG theory. 
\begin{theorem}[\textbf{Main Result}]\label{theo:main-fixed-point}
Assume (A0)-(A12) hold and consider the processes $V(t,p)$, $u^\ast(t)$, $z^o(t)$ and $y(t)$, which are defined in (\ref{eq:prop-mortensen}), (\ref{eq:cont-sep-fixed}) and (\ref{eq:cl-loop-dyn-fp}), respectively. If the constants $(c_1,c_2)$ for (\ref{eq:theo-sens-cont}) and (\ref{eq:lem-sens-meas-cont}) satisfy the gain condition that $c_1c_2<1$ there exists a unique solution to (\ref{eq:fixedpoint}) and hence a unique solution to the MFG system given by (\ref{eq:prop-mortensen}), (\ref{eq:cont-sep-fixed}) and (\ref{eq:cl-loop-dyn-fp}). 
\end{theorem}
\begin{proof}
The proof follows from the Banach fixed point theorem for the map $\bar{\Gamma}\circ\Gamma$ defined on the Polish space $\mathcal{M}_{[0,T]}$ since the gain condition assures that the map is a contraction. 
\end{proof}
\section{$\epsilon$-Nash Equilibrium Property of the MFG Control Laws}\label{sec:e-nash}
Following the common methodology employed in the MFG literature, we shall investigate the performance of the best response control processes obtained in the previous section in a finite population setting. Consider the following dynamics described in (\ref{eq:sta-dyn}):
\begin{eqnarray}
dz_{i}^N(t)&=&\frac{1}{N}\sum_{j=1}^Nf\left(t,z_{i}^N(t),u_{i}^N(t),z_{j}^N(t)\right)dt+\sigma dw_{i}(t), \label{eq:enash-sta-dyn}\\
dy_{i}^N(t)&=&h\left(z_{i}^N(t)\right)dt+ d\nu_{i}(t), \label{eq:enash-obs-dyn}
\end{eqnarray}
with $z_i^N(0)=z_i(0)$, $y_i(0)=0$, $1\leq i \leq N$. Here $\left(w_i(t),\nu_i(t), 1\leq i \leq N\right)$ are independent Brownian motions in $\mathbb{R}$. Similarly, define the following set of MV equations
\begin{eqnarray}
dz_{i}^o(t)&=&f\left[t,z_{i}^o(t),u_{i}^o(t),\mu_t\right]dt+\sigma dw_{i}(t), \label{eq:enash-mv-sta-dyn}\\
dy_{i}^o(t)&=&h\left(z_{i}^o(t)\right)dt+ d\nu_{i}(t), \label{eq:enash-mv-obs-dyn}
\end{eqnarray}
with the same initial conditions where $\mu_t$ is the law of $z_i^o(t)$. Recall that a unique solution exists to (\ref{eq:enash-sta-dyn}) when $u_i^N \in C_{\rm{Lip}{(x)}}\left([0,T]\times \mathbb{R};U\right)$, and a unique consistent solution to (\ref{eq:enash-mv-sta-dyn}) exists when $u_i^o \in C_{\rm{Lip}(p)}\left([0,T]\times \mathsf{E}_k;U\right)$. Notice that in contrast to the coupled process in (\ref{eq:enash-sta-dyn})-(\ref{eq:enash-obs-dyn}), the system in (\ref{eq:enash-mv-sta-dyn})-(\ref{eq:enash-mv-obs-dyn}) gives $N$ decoupled pairs of processes. Let $z(0):=\int_{\mathbb{R}}xdF(x)$ be the mean value of the initial states and define 
\begin{eqnarray}
\epsilon_N:=\left| \int_{\mathbb{R}}x^2dF_N(x)-2z(0)\int_{\mathbb{R}}xdF_N(x)+z(0)^2\right|, \label{eq:def-epsilon}
\end{eqnarray}
where $\lim_{n\rightarrow \infty}\epsilon_N=0$. Consider the following $\sigma$-algebras:
\begin{eqnarray}
\mathcal{F}_t^{(y^o)}&:=&\sigma\big\{y_1^o(s),\ldots, y_N^o(s);0\leq s \leq t\big\},\nonumber\\
 \mathcal{F}_t^{y_i^o}&:=&\sigma\big\{y_i^o(s);0\leq s \leq t\big\},\nonumber\\
 \mathcal{F}_t^{(y^N)}&:=&\sigma\big\{y_1^N(s),\ldots, y_N^N(s);0\leq s \leq t\big\},\nonumber\\
 \mathcal{F}_t^{y_i^N}&:=&\sigma\big\{y_i^N(s);0\leq s \leq t\big\}\label{eq:sigma-algebras}.
\end{eqnarray}
We now define class of admissible control policies. Let 
\begin{eqnarray}
&&\mathcal{U}^o_i:=\bigg\{u(\cdot):u(t)~\mbox{is adapted to}~\mathcal{F}_t^{(y^o)},\nonumber\\
&&~u\in C_{\rm{Lip}(p^N)}\left([0,T]\times \mathsf{E}_k^N;U\right)~\mbox{and}~\mathbb{E}\int_{0}^T|u(t)|^2dt<\infty\bigg\},\nonumber\\
&&\mathcal{U}^{o,d}_i:=\bigg\{u(\cdot):u(t)~\mbox{is adapted to}~\mathcal{F}_t^{y_i^o},\nonumber\\
&&u \in C_{\rm{Lip}(p)}\left([0,T]\times \mathsf{E}_k;U\right)~\mbox{and}~\mathbb{E}\int_{0}^T|u(t)|^2dt<\infty\bigg\},\nonumber\\
&&\mathcal{U}^N_i:=\bigg\{u(\cdot):u(t)~\mbox{is adapted to}~\mathcal{F}_t^{(y^N)},\nonumber\\
&&u \in C_{\rm{Lip}(p^N)}\left([0,T]\times \mathsf{E}^N_k;U\right)~\mbox{and}~\mathbb{E}\int_{0}^T|u(t)|^2dt<\infty\bigg\},\nonumber\\
&&\mathcal{U}^{N,d}_i:=\bigg\{u(\cdot):u(t)~\mbox{is adapted to}~\mathcal{F}_t^{y_i^N},\nonumber\\
&&u \in C_{\rm{Lip}(p)}\left([0,T]\times \mathsf{E}_k;U\right)~\mbox{and}~\mathbb{E}\int_{0}^T|u(t)|^2dt<\infty\bigg\},\label{eq:admissible-controls}
\end{eqnarray}
where $1\leq i \leq N$ and $C_{\rm{Lip}(p^N)}\left([0,T]\times \mathsf{E}_k^N;U\right)$ denote the space of $U$-valued continuous functions on $[0,T]\times \mathsf{E}_k^N$ with Lipschitz coefficients in $\mathsf{E}_k^N:=\otimes_{j=1}^N\mathsf{E}_{k,j}$, where for $1\leq j \leq N$, $\mathsf{E}_{k,j}$ is a copy of $E_k$ 

In the above admissible control policies, $\mathcal{U}^o_i$ represents centralized information on the partially observed states at the infinite population game whereas $\mathcal{U}^N_i$ represents centralized information on the partially observed states at the finite population game. On the other hand $\mathcal{U}^{o,d}_i$ represents decentralized control policies at infinite population game whereas $\mathcal{U}^{N,d}_i$ represents decentralized control policies at the finite population. 

For the dynamic game problem specified in (\ref{eq:enash-sta-dyn}), recall that the cost function for the $i$th agent is given as
\begin{eqnarray}
J_{i}^N(u_{i}^N,u_{-i}^N)=\mathbb{E}\int_{0}^T \frac{1}{N}\sum_{j=1}^NL\left(z_{i}^N(t),u_{i}^N(t),z_{j}^N(t)\right)dt\label{eq:enash-cost}.
\end{eqnarray}
Recall also the generic agent's PO SOCP:

\textit{Generic Agent's SOCP: For $0\leq t \leq T$}
\begin{eqnarray}
dz^o(t)&=&f\left[t,z^o(t),u(t),\mu_t
\right]dt+ \sigma dw(t),~z^o(0)=z(0), \nonumber\\
dy^o(t)&=&h\left(z^o(t)\right)dt+ d\nu(t), ~y(0)=0, \nonumber\\
J(u)&=&\mathbb{E}\int_{0}^TL[z^o(t), u(t), \mu_t]dt\label{eq:enash-socp}
\end{eqnarray}
where $J(u)$ is to be minimized over $\mathcal{U}:=\bigl\{u(\cdot) \in U: u(t)~\text{is}~\mathcal{F}_t^{y^o}-\text{adapted}~\text{and}$ $~\mathbb{E}\int_{0}^T|u(t)|^2dt<\infty\bigr\}$. 

The optimal control law for the above PO SOCP (and hence the best response control process for the MFG game) is characterized by (\ref{eq:cont-sep-fixed}) which we denote by $u^o(t)$. Recall that under (A6)-(A10), $u_i^o \in C_{\rm{Lip}(p)}\left([0,T]\times \mathsf{E}_k;U\right) \subset \mathcal{U}_i^{o,d}$.  
\begin{definition}{\cite{hmc-mfg-main}}\label{def:enash}
Given $\epsilon$, the set of admissible control laws $(u_1^o,\ldots, u_N^o)$ generates $\epsilon$-Nash equilibrium with respect to the cost $J_i^N$ if for any $1\leq i \leq N$
\begin{eqnarray}
J_i^N\left(u_i^o;u_{-i}^o\right)-\epsilon \leq \inf_{u_i \in \mathcal{U}_i^{N}}J_i^N\left(u_i;u_{-i}^o\right)\leq J_i^N\left(u_i^o;u_{-i}^o\right)\label{eq:enash}.
\end{eqnarray}
\end{definition}
We now show that the MFG best responses for a finite $N$ population system (\ref{eq:enash-sta-dyn})-(\ref{eq:enash-mv-obs-dyn}) is an $\epsilon$-Nash equilibrium with respect to the cost function defined in (\ref{eq:enash-cost}). 
\begin{theorem}\label{the:main-enash}
Assume (A0)-(A12) hold and there exists a unique solution to MFG system such that the best response control process $u^o(t,p)$ is continuous in $(t,p)$ and Lipschitz continuous in $p$. Then $\left(u_1^o,u_2^o\ldots, u_N^o\right)$ where $u_i^o=u^o$, $1\leq i \leq N$, generates an $O\left(\epsilon_N+1/\sqrt{N}\right)$-Nash equilibrium with respect to the cost function (\ref{eq:enash-cost}) such that $\lim_{N\rightarrow \infty}\epsilon_N=0$.
\end{theorem}
\begin{proof}
Proof involves several linked perturbation estimates which involve conditional density process. We consider a strategy change for the first agent. Consider the following closed loop individual dynamics under the best response control policies $u_i^o=u^o$, $1\leq i \leq N$, at finite population
\begin{eqnarray}
dz_{i}^{o,N}(t)&=&\frac{1}{N}\sum_{j=1}^Nf\left(t,z_{i}^{o,N}(t),u^{o}\bigl(t,\tilde{\varphi}_i^{o,N}(t)\bigr),z_{j}^{o,N}(t)\right)dt+\sigma dw_{i}(t), \nonumber\\
dy_{i}^{o,N}(t)&=&h\left(z_{i}^{o,N}(t)\right)dt+ d\nu_{i}(t), \label{eq:prf-enash-staobs-dyn}
\end{eqnarray}
where $z_i^{o,N}(0)=z_i(0)$, $y_i^{o,N}(0)=0$ and $\tilde{\varphi}_i^{o,N}(t)$, $1\leq i \leq N$, denote the associated unnormalized filtering processes. Similarly, consider the MV system
\begin{eqnarray}
dz_{i}^o(t)&=&f\left[t,z_{i}^o(t),u^o(t,\tilde{\varphi}_i^o(t)),\mu_t\right]dt+\sigma dw_{i}(t), \label{eq:prf-enash-mv-sta-dyn}\\
dy_{i}^o(t)&=&h\left(z_{i}^o(t)\right)dt+ d\nu_{i}(t), \label{eq:prf-enash-mv-obs-dyn}
\end{eqnarray}
with the same initial conditions and $\tilde{\varphi}_i^o(t)$ denotes the associated unnormalized filtering processes. Following similar lines of the proof of Theorem \ref{the:mv-sde-uniq} and using
\begin{eqnarray}
\left\| \tilde{\varphi}_i^{o,N}(t) - \tilde{\varphi}_i^{o}(t)\right \|_{\mathsf{E}_k}\leq C_f^RC_4 \sup_{0\leq s\leq t}\left| z_{i}^o(t)-z_{i}^{o,N}(t)\right|\label{eq:prf-nash-prf-1}, 
\end{eqnarray}
we obtain, as a consequence of Gronwall's lemma, 
\begin{eqnarray}
\sup_{1\leq j \leq N}\sup_{0\leq t \leq T}\mathbb{E}| z_{j}^o(t)-z_{j}^{o,N}(t)|=O\left(1/\sqrt{N}\right)\label{eq:prf-nash-sup-finvsinf},
\end{eqnarray}
where the right hand side depends on the terminal time $T$. 

Assume now that while each agent $j$, $2\leq j \leq N$, are using the MFG best response control law $u^o(t,p)$, the first agent implement a strategy change from $u^o$ to $u_1 \in \mathcal{U}_1^N$ which yields
\begin{eqnarray}
dz_{1,c}^{o,N}(t)&=&\frac{1}{N}\sum_{j=1}^Nf\left(t,z_{1,c}^{o,N}(t),u_1\bigl(t,\tilde{\varphi}_{1:N,c}^{o,N}(t)\bigr),z_{j,c}^{o,N}(t)\right)dt+\sigma dw_{1}(t), \nonumber\\
dy_{1,c}^{o,N}(t)&=&h\left(z_{1,c}^{o,N}(t)\right)dt+ d\nu_{1}(t), \label{eq:prf-enash-staobs-dyn-per}\\
dz_{i,c}^{o,N}(t)&=&\frac{1}{N}\sum_{j=1}^Nf\left(t,z_{i,c}^{o,N}(t),u^{o}\bigl(t,\tilde{\varphi}_i^{o,N}(t)\bigr),z_{j,c}^{o,N}(t)\right)dt+\sigma dw_{i}(t), \nonumber\\
dy_{i,c}^{o,N}(t)&=&h\left(z_{i,c}^{o,N}(t)\right)dt+ d\nu_{i}(t), \label{eq:prf-enash-staobs-dyn-per-b}
\end{eqnarray}
where initial conditions are given by $z_{i,c}^{o,N}(0)=z_i(0)$, $y_{i,c}^{o,N}(0)=0$ and $\tilde{\varphi}_{i,c}^{o,N}(t)$, $1\leq i \leq N$, denotes the filtering processes and $\tilde{\varphi}_{1:N,c}^{o,N}(t):=\left(\tilde{\varphi}_{1,c}^{o,N}(t),\ldots, \tilde{\varphi}_{N,c}^{o,N}(t)\right)$. 

We also introduce the MV dynamics and its observation:
\begin{eqnarray}
d\hat{z}_{1}^o(t)&=&f\left[t,\hat{z}_{1}^o(t),u_1\big(t,\tilde{\hat{\varphi}}_{1:N}^o(t)\big),\mu_t\right]dt+\sigma dw_{1}(t) \label{eq:prf-enash-mv-sta-dyn-per}\\
d\hat{y}_{1}^o(t)&=&h\left(\hat{z}_{1}^o(t)\right)dt+ d\nu_{1}(t) \label{eq:prf-enash-mv-obs-dyn-per}\\
d\hat{z}_{i}^o(t)&=&f\left[t,\hat{z}_{i}^o(t),u^o(t,\tilde{\hat{\varphi}}_i^o(t)),\mu_t\right]dt+\sigma dw_{i}(t) \label{eq:prf-enash-mv-sta-dyn-per-a}\\
d\hat{y}_{i}^o(t)&=&h\left(\hat{z}_{i}^o(t)\right)dt+ d\nu_{i}(t) \label{eq:prf-enash-mv-obs-dyn-per-a}
\end{eqnarray}
for $2\leq i \leq N$ with the same initial conditions and $\tilde{\hat{\varphi}}_i^o(t)$ denotes the unnormalized filtering equations for the signal and observation pair $\big(\hat{z}_{i}^o(t), \hat{y}_{i}^o(t)\big)$, $1\leq i \leq N$.

It now follows that
\begin{eqnarray}
\sup_{2\leq j \leq N}\sup_{0\leq t \leq T}\mathbb{E}\left| z_{j,c}^{o,N}(t)-\hat{z}_{j}^o(t)\right|=O\left(1/\sqrt{N}\right)\label{eq:prf-nash-sup-b}.
\end{eqnarray}
Furthermore, by the robustness of the nonlinear filter, for $2\leq j \leq N$ and $0\leq t \leq T$,
\begin{eqnarray}
\left \| \tilde{\varphi}_j^{o,N}(t)-\tilde{\hat{\varphi}}_j^o(t)\right\|_{\mathsf{E}_k}\leq \sup_{0\leq s \leq t}C_f^RC_4\left| z_{j,c}^{o,N}(t)-\hat{z}_{j}^o(t)\right|\label{eq:prf-nash-sup-c}.
\end{eqnarray}
Gronwall's lemma, (\ref{eq:prf-nash-sup-b}) and (\ref{eq:prf-nash-sup-c}) imply that
\begin{eqnarray}
\sup_{0\leq t \leq T}\mathbb{E}\left|z_{1,c}^{o,N}(t)-\hat{z}_{1}^o(t)\right|=O\left(1/\sqrt{N}\right)\label{eq:prf-nash-sup-d}.
\end{eqnarray}
Observe that (\ref{eq:prf-nash-sup-d}) and so (\ref{eq:prf-nash-sup-d}) and the robustness of the filtering together imply that
\begin{eqnarray}
\sup_{0\leq t \leq T}\mathbb{E}\left\|\tilde{\hat{\varphi}}_{1}^{o}(t) -\tilde{\varphi}_{1,c}^{o,N}(t) \right\|_{\mathsf{E}_k}=O\left(1/\sqrt{N}\right)\label{eq:prf-nash-sup-e}.
\end{eqnarray}
Consequently, under the modified strategy, we obtain
\begin{eqnarray}
&&J_1^N\left(u_1;u_{-i}^o\right)\equiv \nonumber\\
&&\mathbb{E}\int_{0}^T \frac{1}{N}\sum_{j=1}^NL\left(z_{1,c}^{o,N}(t),u_{1}\left(t,\tilde{\varphi}_{1:N,c}^{o,N}(t)\right),z_{j,c}^{o,N}(t)\right)dt\label{eq:prf-enash-main-1}\\
&&\overset{(\ref{eq:prf-nash-sup-b}),(\ref{eq:prf-nash-sup-d})}{\geq}\mathbb{E}\int_{0}^T\frac{1}{N}\sum_{j=1}^NL\left(\hat{z}_{1}^{o}(t),u_{1}\left(t,\tilde{\varphi}_{1:N,c}^{o,N}(t)\right),\hat{z}_{j}^{o}(t)\right)dt\nonumber\\
&&\hspace{1in}-O\left(\epsilon_N+\frac{1}{\sqrt{N}}\right)\label{eq:prf-enash-main-2}\\
&& \overset{(\ref{eq:prf-nash-sup-b}),(\ref{eq:prf-nash-sup-c})}{\geq}\mathbb{E}\int_{0}^T \frac{1}{N}\sum_{j=1}^NL\bigg(\hat{z}_{1}^{o}(t),u_{1}\bigg(t,\tilde{\varphi}_{1,c}^{o,N}(t),\tilde{\hat{\varphi}}_{2:N}^o(t)\bigg), \hat{z}_{j}^{o}(t)\bigg)dt\nonumber\\
&&\hspace{1in}-O\left(\epsilon_N+\frac{1}{\sqrt{N}}\right)\label{eq:prf-enash-main-3}\\
&& \overset{(\ref{eq:prf-nash-sup-e})}{\geq}\mathbb{E}\int_{0}^T \hspace{-0.05in}\frac{1}{N}\sum_{j=1}^NL\bigg(\hat{z}_{1}^{o}(t),u_{1}\bigg(t,\tilde{\hat{\varphi}}_{1}^{o}(t),\tilde{\hat{\varphi}}_{2:N}^o(t)\bigg),\hat{z}_{j}^{o}(t)\bigg)dt\nonumber\\
&&\hspace{1in}-O\left(\epsilon_N+\frac{1}{\sqrt{N}}\right)\label{eq:prf-enash-main-4}\\
&& \overset{(\ref{eq:mv-conv-partial})}{\geq}\mathbb{E}\int_{0}^T L\big[\hat{z}_{1}^{o}(t),u_{1}\big(t,\tilde{\hat{\varphi}}_{1:N}^o(t)\big),\mu_t\big]dt-O\left(\epsilon_N+\frac{1}{\sqrt{N}}\right)\nonumber
\end{eqnarray}
Furthermore, by the construction of the generic agent's partially observed MFG system (\ref{eq:enash-socp}) we have
\begin{eqnarray}
\mathbb{E}\int_{0}^T L\left[\hat{z}_{1}^{o}(t),u_{1}\left(t,\tilde{\hat{\varphi}}_{1:N}^o(t)\right),\mu_t\right]dt \geq \mathbb{E}\int_{0}^T L\left[z_{1}^{o}(t),u_{1}^o\left(t,\tilde{\varphi}_{1}^o(t)\right),\mu_t\right]dt\label{eq:prf-enash-main-6}.
\end{eqnarray}
On the other hand we have
\begin{eqnarray}
&&\mathbb{E}\int_{0}^T L\left[z_{1}^{o}(t),u_{1}^o\left(t,\tilde{\varphi}_{1}^o(t)\right),\mu_t\right]dt\nonumber\\
&&\overset{(\ref{eq:mv-conv-partial})}{\geq}\mathbb{E}\int_{0}^T \frac{1}{N}\sum_{j=1}^NL\left[z_1^{o}(t),u_{1}^o\left(t,\tilde{\varphi}_{1}^o(t)\right),z_j^o(t)\right]dt-O\left(\epsilon_N+\frac{1}{\sqrt{N}}\right)\nonumber\\
&&\overset{(\ref{eq:prf-nash-sup-finvsinf})}{\geq}\mathbb{E}\int_{0}^T \frac{1}{N}\sum_{j=1}^NL\left[z_1^{o,N}(t),u_{1}^o\left(t,\tilde{\varphi}_{1}^{o,N}(t)\right),z_j^{o,N}(t)\right]dt-O\left(\epsilon_N+\frac{1}{\sqrt{N}}\right)\nonumber\\
&&\equiv J_1^N\left(u_1^o;u_{-1}^o\right)-O\left(\epsilon_N+\frac{1}{\sqrt{N}}\right)\label{eq:prf-enash-main-7}.
\end{eqnarray}
It now follows from (\ref{eq:prf-enash-main-1})-(\ref{eq:prf-enash-main-7})
\begin{eqnarray}
J_1^N\left(u_1^o;u_{-1}^o\right)-O\left(\epsilon_N+\frac{1}{\sqrt{N}}\right)\leq \inf_{u_1 \in \mathcal{U}_1^N}J_1^N\left(u_1;u_{-1}^o\right), \label{eq:prf-enash-main-8}
\end{eqnarray}
which completes the proof for agent $1$. The analysis for the other agents follows similarly. 
\end{proof}
\section{An Explicit Example with Finite Dimensional Nonlinear Filters}\label{sec:example}
In nonlinear filtering theory, initiated by the work \cite{benes}, there has been a considerable progress in representing a large class of nonlinear filters with finite dimensional quantities. Furthermore, in the case that these quantities are sufficient for the control, then the infinite dimensional conditional density can be replaced by this finite dimensional sufficient statistics. Consequently, a rigorous proof of the verification theorem can be obtained in this finite dimensional, completely observed stochastic optimal control problem. The literature for such a theory is vast and we refer reader to \cite[Section I]{bambos-1} for a succinct summary of the related works. In this section we consider such an explicit model whose nonlinear filters can be expressed with finite dimensional quantities and hence yield a more tractable PO MFG system. 

Consider the following MFG system where $f$, $h$ and $L$ are motivated from \cite[Section III.B-C]{bambos-1} and given in the following form: Let $z_i^N(t)=[z_{i,1}^N(t), z_{i,2}^N(t)]^T$, $G_t:=\begin{bmatrix} G_t^{11} \hspace{-0.03in}&0\\0\hspace{-0.03in}&G_t^{22}\end{bmatrix}$, $w_i(t):=(w_{i,1}(t),w_{i,2}(t))$ and
\begin{eqnarray}
dz_{i}^N(t)&=&\big[g\big(t,z_{i,1}^N(t)\big), u_i^N(t,y_i)\big]^Tdt+G_tdw_i(t)\nonumber\\
\hspace{-0.1in}dy_i(t)&=&H_tz_{i}^N(t)+N_t^{\frac{1}{2}}db_i(t), \nonumber\\
z_{i}^N(0)&=&z_{i}(0), y_i(0)=0,~1\leq i \leq N,\label{eq:sta-obs-dyn-ex}
\end{eqnarray}
where (i) $z_{i,j}^N(t):[0,T]\rightarrow \mathbb{R}$, $j\in\{1,2\}$, (ii) $(w_{i,1}(t),w_{i,2}(t)$, $b_i(t);0\leq t \leq T)$ are independent Brownian motions in $\mathbb{R}$ which are also independent of $z_i(0)$, (iii) $g$, $H_t$, $G_t$, $N_t$ and $h$ satisfy \cite[Assumptions A2)-A6) and A9)]{bambos-1}. We note that the example is a suitably specialized case of \cite[Section III.B-C]{bambos-1} where (A2)-(A4) and (A5) are satisfied by \cite[Assumptions A2), A4), A9)]{bambos-1}. The cost function is given by
\begin{eqnarray}
J_{i}^N(u_i^N;u_{-i}^N):=
\mathbb{E} \int_{0}^T \frac{1}{N}\sum_{j=1}^N l_2\left(t,z_i^N(t),u_i^N(t,y_i), z_{j}^N(t)\right)dt\label{eq:cost-ex}
\end{eqnarray}
where for any $z \in \mathbb{R}^2$, $l_2(\cdot,z)$ satisfy \cite[Assumption A7]{bambos-1}. Let $\hat{\mathcal{U}}_i:=\{u_i(\cdot) \in L_{\mathcal{F}_t^{y_i}}^2\left([0,T];U\right)\}$.

For the dynamics (\ref{eq:sta-obs-dyn-ex}), we also consider the limiting system
\begin{eqnarray}
dz_{i}(t)&=&\big[g\big(t,z_{i,1}(t)\big), u_i(t,y_i)\big]^Tdt+G_tdw_i(t)\nonumber\\
dy_i(t)&=&H_tz_{i}(t)+N_t^{\frac{1}{2}}db_i(t), \nonumber\\
J_{i}&:=&\mathbb{E} \int_{0}^T l_2\left[t,z_i(t),u_i(t,y_i), \mu_t\right]dt\nonumber\\
z_{i}(0)&=&z_{i}(0), y_i(0)=0,~1\leq i \leq N,\label{eq:sta-obs-dyn-mv-ex}
\end{eqnarray}
where we emphasize that the mean field exists only in the performance functions. Assume now that there exists a function $\phi(t,x) \in C_{t,x}^{1,2}\left([0,T]\times\mathbb{R}\right)$ which satisfies
\begin{eqnarray}
\partial_t \phi(t,x)+\frac{1}{2}(G_t^{11})^2\partial^2_{xx}\phi(t,x)+\frac{1}{2}|G_t^{11}\partial_{x}\phi(t,x)|^2=\frac{1}{2}\left(Q_tx^2+2m_tx+\delta_t\right)\label{eq:phi-ex}
\end{eqnarray}
where $Q$, $m$ and $\delta$ are arbitrary functions. Then, for each $u \in \hat{\mathcal{U}}$ with the following conditions are satisfied:
\begin{itemize}
\item [(E1)] The random variable $z_i(0)$, $1\leq i \leq N$, has density
\begin{eqnarray}
q_i(z_i)=\frac{\exp\big(-P_0^{-1}\left(z_i-\xi\right)^2\big)}{(2\pi)^{\frac{n}{2}}|P_0|^{\frac{1}{2}}},~P_0\geq 0\label{eq:inident-exm}
\end{eqnarray}
for a given $\xi>0$. 
\item [(E2)] The nonlinear drift function satisfies
\begin{eqnarray}
g(t,x)=(G_t^{11})^2\partial_x\phi(t,x),\label{eq:as2-exm}
\end{eqnarray}
\end{itemize}
the system defined by (\ref{eq:sta-obs-dyn-mv-ex}) has an information state given by
\begin{eqnarray}
q_t^i=\exp\left(\phi(x,t)+\lambda_t\right) \frac{\exp\big(-\frac{1}{2}\big| P_t^{-\frac{1}{2}}\left(z_i-r_t\right)^2\big|\big)}{(2\pi)^{\frac{n}{2}}|P_t|^{\frac{1}{2}}}\label{eq:fin-inf-state}
\end{eqnarray}
where $P_t\geq 0, \forall t \in [0,T]$ and $\lambda_t \in \mathbb{R}$ and $r_t \in \mathbb{R}^2$ are given by:
\begin{eqnarray}
dr_t&=&\left(P_tQ_t\right)r_tdt-P_tm_tdt+u_i(t)dt, ~r(0)=\xi\label{eq:fin-dim-filt}\\
\frac{dP_t}{dt}&=&-P_tQ_tP_t+G_tG_t^T, ~P(0)=P_0\label{eq:ricatti}\\
\lambda_t &=&\frac{1}{2}\int_{0}^t\left(Q_sr_s^2+2m_sr_s+\delta_s+\mbox{tr}(P_sQ_s\right)ds\nonumber.
\end{eqnarray}
For the partially observed MFG problem defined by (\ref{eq:sta-obs-dyn-ex})-(\ref{eq:cost-ex}), the information state $q_t$ in (\ref{eq:fin-inf-state}) for the generic agent will be given in a finite-dimensional form if the PDE in (\ref{eq:phi-ex}) has shown to an explicit solution. This is discussed in \cite[Theorem 3.9]{bambos-1} and it is shown that solution is given by
\begin{eqnarray}
\phi(t,x)&=&\log\Gamma(t,x) \nonumber\\
\Gamma(t,x)&=&\frac{1}{2}\Delta_t x^2+ x\varsigma_t+\eta_t\label{eq:gamma-ex}
\end{eqnarray}
where $\Delta(\cdot), \varsigma(\cdot), \eta(\cdot)$ are given in the statement of the theorem which sets the function $g(t,x)$ to be:
\begin{eqnarray}
g(t,x)&=&\frac{G_t^{11}G_t^{11}}{\Gamma(t,x)}\big(\Delta_tx +\varsigma_t\big)\label{eq:gh-fin}.
\end{eqnarray}
Consequently, by (\ref{eq:sta-obs-dyn-ex})-(\ref{eq:gh-fin}), we obtain a PO MFG model whose filtering equation has a finite dimensional solution. 
\begin{remark}
Consider the following filtering problem
\begin{eqnarray}
dz(t)&=&f\big(z(t)\big)dt+dw(t)\label{eq:rem-fin-1}\\
dy(t)&=&\int_{0}^t z(s)ds+d\nu(t)\label{eq:rem-fin-2}.
\end{eqnarray}
It is shown in \cite{benes} that the solution to the unnormalized conditional density of $\rho(t,x)$ of $P(z(t)|\mathcal{F}_t^y)$ is given in terms of 10 sufficient statistics when $f$ satisfies the condition that $f'+f^2=$ $ax^2+bx+c$, $a\geq -1$. The model defined in (\ref{eq:sta-obs-dyn-ex}) is a generalization of this result to the case where the control enters into the dynamics which is in general the case in MFG. 
\end{remark}
By the Separation Theorem discussed in Section \ref{sec:socp-gen}, the optimal control process for a generic agent is given by the minimizer of a Hamiltonian which was defined earlier. Before, we define the forward and backward operators as:
\begin{eqnarray}
\mathcal{O}_t^a&:=&\frac{1}{2}\nabla^2_{xx}+\left(g, a\right)^T\nabla_x \nonumber\\
\mathcal{O}_t^{\ast a}&:=& \frac{1}{2}\nabla^2_{xx}-\left(g,a\right)^T\nabla_x-\mathbf{tr}\left[\nabla_x \left(g,a\right)\right]\label{eq:fwd-bck-op-ex}.
\end{eqnarray}
Consequently, one can write the the optimal control in the form:
\begin{eqnarray}
u^\ast &=&\{u^\ast(t,p)=a^\ast\left(T-t, q_t\right);0 \leq t \leq T\}\nonumber\\
a^\ast\left(\tau, p\right)&=&\arg\min_{a \in U} \left\{\left(\mathcal{O}_{T-t}^a\mathsf{D}V(\tau,p)(\cdot),p\right)+\left(l_2[t,\cdot, a, \mu_t],p\right)\right\}\label{eq:cont-sep-ex}.
\end{eqnarray}
Notice that the control law given in the separated form $a^\ast\left(T-t, q_t\right)$ depends on $q_t$ which has an explicit solution given by (\ref{eq:fin-inf-state})-(\ref{eq:gamma-ex}) and entails a finite dimensional representation. In other words, the best response process of the agents in the above partially observed MFG example, which is $\mathcal{F}_t^{y}$-adapted, can be written as a function of $\big(G_t, g, y_t, H_t,\phi(t),P_0,\xi,Q_t, C_t^{-1}, m_t\big)$. 

Finally, we remark that one can follow the rest of the analysis in Section \ref{sec:fixedpoint} in order to obtain a sufficient condition so that PO MFG system admits a unique solution. Alternatively, one can employ the models considered in \cite[Section IV]{bambos-1} and obtain an equivalent LQG MFG system for such nonlinear models. Such a path is currently under investigation and we will report further details in a future work. 

\section{Comparison with MFG with a Partially Observed Major Agent}\label{sec:comp-maj}
In the nonlinear MFG setup where there is a major agent, the best response control policy of each minor agent depends on the state of the major agent in addition to the mean field which is stochastic; the mean field is adapted to the filtration generated by the Brownian motion of the major agent \cite{NP-Siam2013} (see also \cite{minyi-major} for the LQG case). More explicitly, consider the following stochastic coefficient MV (SMV) type dynamics:
\begin{eqnarray}
dz_{0}(t)&=&f_{0}[t,z_{0},u_{0}(t,\omega,z_{0}),\mu_t(\omega)]dt+\sigma_{0}[t,z_{0},\mu_t(\omega)]dw_{0}(t,\omega),\label{eq:stmfg-maj-sta}\\
dz(t)&=&f[t,z,u(t,\omega,z),\mu_t(\omega)]dt+\sigma[t,z(t),\mu_t(\omega)]dw(t),\label{eq:stmfg-min-sta}
\end{eqnarray} with $z_{0}(0)=z_{0}(0)$ and $z(0)=z(0)$ where $\big(\mu_t(\omega)\big)_{0\leq t \leq T}$ satisfies $P\left(z(t)\leq \alpha|\mathcal{F}_t^{w_{0}}\right)=\int_{-\infty}^\alpha\mu_{t}(\omega,dx)$ for all $\alpha \in \mathbb{R}^n$. The SMV SDEs in (\ref{eq:stmfg-maj-sta}) and (\ref{eq:stmfg-min-sta}) represent the closed loop state dynamics (with $\mathcal{F}_t^{w_0}:=\sigma\{w_0(s):0\leq s \leq t\}$-adapted feedback control) of the major and the generic minor agents, respectively, in the infinite population limit. Let $\mathcal{U}_0:=\{u(\cdot)\in U_0:u ~\mbox{is adapted to }\mathcal{F}_t^{w_0}~\mbox{and}~
\mathbb{E}\int_{0}^T|u(t)|^2dt<\infty\}$ and $\mathcal{U}_i:=\{u(\cdot)\in U : u ~\mbox{is adapted to }$ $\mathcal{F}_t^{w_0,w_i}~\mbox{and}~
\mathbb{E}\int_{0}^T|u(t)|^2dt<\infty\}$. Then we define two SOCPs as follows:
\begin{itemize}
\item [] \hspace{-0.3in}\textit{Major Agent's SOCP at Infinite Population:}
\begin{eqnarray}
&&d z_0(t) = f_0[t, z_0(t),u_0(t),\mu_t(\omega)]dt +\sigma_0[t, z_0(t),\mu_t(\omega)] dw_0(t), \nonumber\\
&&J_{0}(u_0) = \mathbb{E} \int_0^T L[z_0(t), u_0(t),\mu_t(\omega)]dt.\nonumber
\end{eqnarray}
\item [] \hspace{-0.3in}\textit{Generic Minor Agent's SOCP at Infinite Population:}
\begin{eqnarray}
&&d z_i(t) = f[t, z_i(t),u(t),\mu_t(\omega)]dt +\sigma[t, z_i(t),\mu_t(\omega)] dw_i(t) \nonumber\\
&&J_{i}(u) = \mathbb{E} \int_0^T \hspace{-0.05in}L[z_i(t),z_0(t),u(t),\mu_t(\omega)]dt.\nonumber
\end{eqnarray}
\end{itemize}
where $J_0$ is to be minimized over $\mathcal{U}_0$, $J_i$ is to be minimized over $\mathcal{U}_i$, $\left(w_i(t);1\leq i \leq N\right)$ are $N$ independent Brownian motions and $u_0(t)$ and $u(t)$ are $\mathcal{F}_t^{w_0}$-adapted best response control processes (i.e., the unique solutions satisfying \cite[(5.14)-(5.19)]{NP-Siam2013}. The main contribution of the nonlinear MM-MFG theory can then be summarized by that the set of control laws given by $u_{0}^N=u_0$, $u_{i}^N=u, i=1,\ldots, N$ when applied to a finite population game generates an $\epsilon$-Nash equilibrium \cite[Theorem 7.2]{NP-Siam2013}. One essential point in this result is that the state of the major agent and the stochastic measure induced by the generic minor agent, $\left(z_{0}(t,\omega),\mu_t(\omega)\right)_{0\leq t \leq T}$, are completely observed by the minor agents. It is worth remarking that the solution to these two SOCPs are given by certain backward SPDEs (BSPDE) (see \cite[Theorem 4]{sen-caines-po-sicon}). This nonstandard feature is due the fact that the dynamics and the cost functions have explicit dependence on the underlying probability space, i.e., each $f_0$, $f$, $L_0$ and $L$ is $\mathcal{F}_t^{w_0}$-adapted. Such problems are referred to as SOCP with random parameters which is introduced in \cite{peng}. 

Motivated by the observation that the best response control policies depend upon the state of the major agent, the corresponding PO MM-MFG is examined in \cite{sen-caines-po-sicon} where it is assumed that minor agents have partial observations on the major agent's state in a distributed manner where each agent has complete observation on its own state. 
\section{Conclusions}\label{sec:conc}
In this work we consider an MFG model where the agents have nonlinear dynamics and cost functions and have noisy measurements on their individual state dynamics. By constructing the associated completely observed model via an application of nonlinear filtering theory and the Separation Principle, a control problem with infinite dimensional state space is formulated and a solution is characterized via a solution to the HJB equation in this function space. The optimal control law obtained from the solution of the HJB equation is next applied by the agents in the infinite population limit. We show that the aggregate behaviour of the population under such policies collectively generate the mean field and then establish the $\epsilon$-Nash property of such solutions.
\section{Acknowledgement}
The authors wish to thank the referees for their valuable comments which greatly improved the quality of the manuscript.
\bibliographystyle{siamplain}
\bibliography{sencainespomfgsicon}




\end{document}